\documentclass[12pt
]{amsart}   
\usepackage{geometry}                	

\usepackage{amssymb}
\usepackage{amsmath}
\usepackage{mathtools}
\usepackage{color}
\usepackage{xcolor}
\usepackage[normalem]{ulem}
\usepackage{psfrag}
\usepackage{wrapfig}

\usepackage{amsfonts,amsthm}
\usepackage{microtype}

\def\Nn{{\mathbb N}} 
\def\Zz{{\mathbb Z}} 
\def\Qq{{\mathbb Q}} 
\def\Rr{{\mathbb R}} 

\def\Aa{{\mathcal A}} 

\def\genset{\ensuremath{\mathcal S}} 
\usepackage[T1]{fontenc}
\def\deriv{{\text{\rm\dh}}}

\def\dispG{\deriv{\match}}
\def\dispP{\deriv{P}}

\def\matchpop{\match_{\popfunc}}
\def\match{{m}}

\DeclareMathOperator\state{\sf state} 
\DeclareMathOperator\dist{d} 

\def\popfunc{{\wp}}
\def\local#1{\deriv #1}

\def\nbhd{{\mathcal N}}
\def\sft{\Omega}
\def\sftS{\sft_S}
\def\sftP{\sft_P}
\def\sftPre{\sft_0}

\def\density{\Delta}

\def\To{{\rightarrow}}

\def\Mm{{\mathcal M}} 

\def\mioi{{[-1..1]}}

\def\interval{{I}}
\def\infint{{I_\infty}}
\def\infint{{\Zz_{\geq 0}}}

\def\Aamax{{\Aa_\text{max}}}
\def\Aabig{{\Aa_\text{big}}}
\def\Aamin{{\Aa_\text{min}}}
\def\mubig{{\mu_{\text{big}}}}

\def\proj{{\Pi}}
\def\tilga{{\tilde{\gamma}}}
\def\gap{{\gamma'}}

\def\tta{{\tt a}}
\def\ttb{{\tt b}}
\def\ttc{{\tt c}}

\parskip =\baselineskip

\newtheorem{thm}{Theorem}[section]
\newtheorem{lem}[thm]{Lemma}
\newtheorem{prop}[thm]{Proposition}
\newtheorem{definition}[thm]{Definition}
\newtheorem{cor}[thm]{Corollary}
\newtheorem{clm}[thm]{Claim}

\theoremstyle{definition}
\newtheorem{rmk}[thm]{Remark}

\newenvironment{sumsubscripts}{{\fontdimen16\textfont2=5pt
\fontdimen17\textfont2=5pt}} {{\fontdimen16\textfont2=3pt
\fontdimen17\textfont2=3pt}}

\def\pparagraph#1{{\bf #1}}

\title[strongly aperiodic SFTs on hyperbolic groups]{Strongly aperiodic subshifts of finite type \\ on hyperbolic groups}

\author{David B. Cohen}
\address{The University of Chicago,
Department of Mathematics,
5734 S University Ave,
Chicago, IL 60637}
\email{davidbrucecohen@gmail.com} 
\author{Chaim Goodman-Strauss}
\address{Department of Mathematical Sciences,  University of Arkansas, Fayetteville, AR 72701} 
\email{strauss@uark.edu} 
\author{Yo'av Rieck}
\address{Department of Mathematical Sciences,  University of Arkansas, Fayetteville, AR 72701} 
\email{yoav@uark.edu}

\date{\today}

\definecolor{gray}{gray}{0.6}

\begin{document}
\maketitle

This paper is devoted to proving the following theorem.

{\bf Theorem.} {\em A hyperbolic group admits a strongly aperiodic subshift of finite type if and only if it has 
at most one end.}

 We introduce the subject in Section~\ref{Section:Introduction} and give an informal outline in Section~\ref{Section:Outline}. In
 Section~\ref{section:Definitions}, we formally define our terms and set up the proof, which is a combination of the results of 
 Sections~\ref{section:Definitions}--\ref{subsection:aperiodicity} as follows:

\begin{proof}[Proof of the Theorem] Propositions~\ref{lemma:SFTP},~\ref{Cor:PopShellingsExist}, and~\ref{Prop:Aperiodic} show that any one-ended hyperbolic group $G$ 
admits a non-empty subshift of finite type in which no configuration  has an infinite order stabilizer. 
By Proposition~\ref{prop:NoTorsion}, $G$ admits a subshift of finite type in which no configuration
has a stabilizer of finite order. Proposition~\ref{prop:CombiningSFTs} shows that the product of 
these subshifts is a strongly aperiodic subshift of finite type on $G$.

By Proposition~\ref{prop:NoTorsion} every zero-ended group (that is, every finite group) 
admits a strongly aperiodic subshift of finite type, 
and~\cite{DavidCohen17} shows no group with two or more ends
admits such a subshift.
\end{proof}

  \section{Introduction}\label{Section:Introduction}

 Loosely speaking, a strongly aperiodic subshift of finite type on a group $G$ is given by a finite set of local rules for decorating $G$, so that all global symmetry is destroyed. That is, a finite collection of locally checkable rules 
 ensures that any pair of points have finite  neighborhoods that are decorated distinctly. In many settings, such as on $\Zz^2\subset\Rr^2$, subshifts of finite type are essentially the same phenomenon as matching rule tiling spaces, which are each determined by a given finite set of marked-up tiles, such as the Penrose tiles.\footnote{In any appropriate setting, each subshift of finite type can be interpreted as a matching rule tiling space, each configuration in the subshift being a tiling in the tiling space. The converse is not necessarily the case.~\cite{Radin94}} The two areas arose in different ways but soon became linked:

H. Wang~\cite{Wang61} interpreted remaining cases of Hilbert's {\em Entscheidungsproblem}  in the foundations of logic as being about how square tiles with marked edges might fit together in $\Zz^2$.  As an aside, Wang asked whether one can algorithmically decide 
 the ``domino tiling problem'': {\em Can a given finite set of tiles be used to form a tiling?}

Wang pointed out that if (in his or any appropriate setting) the tiling problem were in fact undecidable, then there must exist aperiodic sets of tiles.\footnote{
If there were not an aperiodic set of tiles, every set of tiles would either not tile the plane (and so have some maximum sized disk that can be tiled) or would admit a periodic tiling (and so have some finite fundamental domain). By enumerating finite configurations, one eventually determines which, deciding the problem. Note that the undecidability of the tiling problem in fact implies {\em weak} aperiodicity, but as it happens, there is no distinction in the Euclidean plane. } Soon R. Berger proved the tiling problem undecidable in $\Zz^2$ and gave the first examples of these aperiodic sets.~\cite{Berger66,Robinson71} 

The ``tiling problem'' for SFTs on a group $G$ asks whether a given set of local rules determine a nonempty subshift, that is whether there exists  a decoration of $G$ satisfying the local rules. Berger's result showed that this problem is undecidable when $G=\Zz^2$.

\pparagraph{Subshifts.}   
Given a finite  set of ``markings'' $\Aa$, the set $\Aa^G$ consists of all possible ways to mark $G$ 
by  $\Aa$ (we give precise definitions in Section~\ref{subsectionSFTs}). Equipped with the product topology and the $G$-action given by shifting coordinates, 
$\Aa^G$ is known as the {\em full shift} on $G$ and  its closed $G$-invariant subsets are known as {\em subshifts}.
We refer to elements of an SFT as {\em configurations}.  
Subshifts are an essential tool in the study of dynamical systems;
every $0$-dimensional expansive
system on $G$ is a subshift \cite[Proposition 2.8]{CoornaertPapadopoulosBook}, every expansive system is a factor of a subshift \cite[Proposition 2.6]{CoornaertPapadopoulosBook}, and if $G$ is nonamenable, 
a theorem of Seward~\cite[Theorem 1.2]{SewardNonamenable} 
shows that every topological dynamical system over $G$ is a factor of a subshift.

 \pparagraph{Subshifts of finite type.}  A subshift of finite type (SFT) is a subset of $\Aa^G$ obtained by ``forbidding'' (or, equivalently, ``allowing'') some finite set of patterns. A pattern is a function from some finite $F\subseteq G$ to $\Aa$. We say that a pattern $p:F\To\Aa$ appears in $\omega:G\To\Aa$ if there is some $g\in G$ such that $\omega(gf)=p(f)$ for all $f\in F$.
 That is, given some finite collection $\mathcal{F}$ of forbidden patterns, if $\Omega\subset\Aa^G$ consists of all $\omega$ in which no $p\in\mathcal{F}$ appears, then $\Omega$ is said to be an SFT. 
 For example, if $\Aa=\{0,1\}$ and $G=\Zz$, the set of all $\omega\in\Aa^\Zz$ such that $(\omega(n),\omega(n+1))$ is never equal to $(1,1)$ forms an SFT. As expected, SFTs are subshifts, and while the converse is false, every subshift can be obtained by forbidding some (usually infinite) set of patterns.

\pparagraph{Weak aperiodicity} was not recognized until after Mozes' definition of {\em strong aperiodicity} in~\cite{Mozes97}, in which he gives examples  of both kinds. An SFT is weakly aperiodic if it is non-empty and the $G$-orbit of every configuration is infinite --- that is, an infinite subgroup of $G$ is allowed to fix a configuration, provided it has infinite index.

Similarly,  a set of tiles is weakly aperiodic if it is possible to form a tiling with congruent copies of them, but never a tiling with a  compact fundamental domain. However, as suggested in the figure on page~\pageref{pageWithGrowthPicture}, such tiles might admit a tiling with an infinite cyclic symmetry. 

In hindsight, weak aperiodicity had  often appeared earlier --- indeed, in a given setting, it is weak aperiodicity that is implied by the undecidability of the tiling problem.  
Block and Weinberger constructed a weakly aperiodic tile set for any nonamenable cover of a compact Riemannian manifold \cite{BlockWeinberger92}.
In the setting of hyperbolic groups, weakly aperiodic SFTs were constructed by 
Gromov~\cite[\S 7.5, 7.6, 8.4]{Gromov87} and Coornaert and  Papadopoulos~\cite{CoornaertPapadopoulosBook};
these SFTs exist on  any  
hyperbolic group but are \em never \em strongly aperiodic.

S. Mozes~\cite{Mozes97} gave  weakly aperiodic tilings on   rank-1 symmetric spaces, by decorating tiles shaped like the fundamental domain of one lattice with information about how it may interact with the tiling by fundamental domains of another,  incompatible lattice, and 
applying Mostow rigidity to prove weak aperiodicity.

\pparagraph{Strong aperiodicity.} On the other hand, an SFT  is said to be strongly aperiodic if it is nonempty and the $G$-action upon it  is free, meaning that no element of $G$ fixes any configuration (some authors allow configurations with finite stabilizers). 
Similarly, a set of tiles is strongly aperiodic if it does admit a tiling, but only tilings that have no symmetry whatsoever
 (some authors allow tilings with finite symmetry).

 Wang himself conjectured that aperiodicity (of any kind) was absurd, but the first strongly aperiodic sets of tiles soon appeared in $\Zz^2$~\cite{Berger66,Robinson71} and many others have been found  since, mostly based on R. Berger's initial use of hierarchically arranged structures~\cite{Mozes89,CGS98, FerniqueOllinger10}, or the theory of quasicrystals stemming from  N.G. De Bruijn's higher dimensional analogue of Sturmian sequences~\cite{DeBruijn81A, DeBruijn81B}.  J. Kari gave a third model~\cite{Kari96},  
which was adapted to give the first strongly aperiodic tilings of ${\mathbb H}^n$~\cite{CGS2005}. We will give a list of groups known to have strongly aperiodic SFTs momentarily, but first we survey groups known not to have such subshifts.

\pparagraph{Obstructions to the existence of a strongly aperiodic SFT.} 
To see that $\Zz$ has no strongly aperiodic SFT, let $\Omega\subset\Aa^\Zz$ be a nonempty SFT, and consider any $\omega\in\Omega$. Because there are only finitely many possible words of a given length in $\Aa$, we see that $\omega$ contains a subword of the form $uvu$ for some words $u$ and $v$ which are longer than all of the forbidden patterns used to define $\Omega$. But then it is easy to check that $\ldots uvuvuv\ldots$ defines a periodic configuration in $\Omega$.
This was extended to all free groups by~\cite{piantadosi}.

The above argument was  generalized by Cohen~\cite{DavidCohen17}, who showed that any group $G$ with at least two ends admits no strongly aperiodic SFT.
Additionally, Jeandel~\cite[Proposition 2.5]{jeandel} has shown that any  recursively presented group with undecidable word problem does not admit a strongly aperiodic SFT.
These are the only known obstructions and we naturally ask:

{\bf Question: }{\em 
Does there exist a one ended finitely generated group with decidable word problem that does not
admit a strongly aperiodic SFT?}

\pparagraph{Groups known to have a strongly aperiodic SFT.}  Whether or not a group admits  a strongly aperiodic SFT is a 
quasi-isometry invariant under mild conditions \cite{DavidCohen17}, and a commensurability invariant \cite{CarrollPenland}.  

\begin{itemize}
\item As above, Berger \cite{Berger66} showed that $\Zz^2$ itself  admits a strongly aperiodic SFT. More generally, \cite{jeandelpolycyclic} has shown that polycyclic groups admit strongly aperiodic subshifts of finite type.
\item Work of Barbieri and Sablik \cite{barbierisablik} shows that any group of the form $\Zz^2 \rtimes H$, where $H$ has decidable word problem, admits a strongly aperiodic SFT. This is a very broad collection of groups since it includes $\Zz^2\times H$ for any $H$ with decidable word problem, as well as the group $\text{Sol}^3\cong\Zz^2\rtimes \Zz$.
\item Work of Mozes implies that uniform lattices in simple Lie groups of rank at least 2 admit strongly aperiodic SFTs. \cite{Mozes97}
\item Work of Jeandel shows that, $\Zz\times T$ admits a strongly aperiodic SFT, where $T$ denotes Thompson's group $T$. (In fact, Jeandel shows that $\Zz\times H$ admits a strongly aperiodic SFT whenever $H$ acts on the circle in a way which satisfies certain dynamical conditions.) \cite{jeandel}
\item Work of the first two authors \cite{CohenGS} shows that the fundamental groups of hyperbolic surfaces admit strongly aperiodic SFT.
\item Barbieri shows that the direct product of any three  infinite, finitely generated groups with decidable word problem admits a strongly aperiodic SFT; the Grigorchuck group is an example~\cite{BarbieriArXiv}.
\end{itemize}

Note that, with the exception of surface groups, all known examples of strongly aperiodic SFTs are on groups which have direct product of infinite groups as a subgroup.\footnote{Though in ${\mathbb H}^n$ there are constructions of strongly  aperiodic sets of tiles, these do not give rise to SFTs on lattices.} 
There remain many naturally occurring groups, including mapping class groups, ${{\textrm{Out}F_n}}$, some Coxeter groups, and non-uniform lattices in higher rank (like $SL(n;\Zz)$), for which it is unknown whether strongly aperiodic SFTs exist. In this paper we address the case of hyperbolic groups.

\pparagraph{Hyperbolic groups.} 
Hyperbolic groups are groups whose Cayley graphs satisfy a geometric ``slim triangles'' condition which holds in hyperbolic space (see Section~\ref{subsectionHyperbolicGroups} for definitions). These groups are quite well behaved---for example, they are always finitely presented and have decidable word problem. The class of hyperbolic groups includes fundamental groups of closed hyperbolic manifolds, free groups, so-called ``random groups'' (with high probability), groups satisfying certain geometric small cancellation conditions, and many Coxeter groups.

Groups acting discretely on hyperbolic space have been studied for over a century. M. Dehn~\cite{Dehn} constructed \em Dehn's algorithm \em
to decide the word problem in surface groups, where  
by Dehn's algorithm we mean  any rewriting system that shortens a given word monotonically, ending with the empty word exactly when the given word represents the identity of the group.   
Another classic property of surface groups is that their growth rate is exponential.
These were slowly generalized:
In 1968 J. Milnor~\cite{Milnor} showed that under certain negativity assumptions on the curvature of a closed manifold, the growth rate of its fundamental group is  exponential.  Cannon studied geometric and algorithmic properties of discrete subgroups of hyperbolic isometries~\cite{Cannon1984,Cannon1991}, laying the groundwork for shortlex automata soon implicit in~\cite{Gromov87} 
and taking center stage in~\cite{wpig}. Finally Gromov~\cite{Gromov87} defined \em hyperbolic 
groups\em, which include fundamental groups of closed negatively curved manifolds, 
showing that they have exponential growth and are the \em only \em groups in which Dehn's algorithm can be used.

 Our main theorem gives a strongly aperiodic subshift on  any one-ended hyperbolic groups; this resolves the question above for all hyperbolic groups.
Note that no hyperbolic group contains a product of infinite groups.

\section{Outline of the proof}\label{Section:Outline}
In this section we give an informal overview to facilitate reading the paper. In Section~\ref{section:Definitions} we define our terms more precisely.

Suppose that $G$ is a one ended hyperbolic group. It is not hard to show (Proposition~\ref{prop:NoTorsion}) that $G$ admits an SFT where no finite order element fixes a configuration. Our main goal is thus to find an SFT where no infinite order element fixes a configuration, since by Proposition~\ref{prop:CombiningSFTs} we could then take a product of these subshifts and obtain a strongly aperiodic SFT on $G$. This goal will be fulfilled by the populated shellings defined in \S\ref{Section:PopulatedShellings}, or more precisely, by the set of all local data associated to populated shellings. Roughly speaking, this attack combines two key ideas from the literature.

\begin{itemize}
\item Shortlex shellings, defined in \S\ref{section:shortlexShellings}, are inspired by the SFTs used in \cite{CoornaertPapadopoulosBook} and \cite{Gromov87} to ``code'' the boundary of a hyperbolic group. The set of local data of shortlex shellings forms a nonempty SFT for which the stabilizer of every configuration is virtually cyclic.
\item Incommensurability of growth rates is the key tool used in \cite{CohenGS} to ``kill'' infinite cyclic periods on certain subshifts on surface groups (by decorating these subshift with extra data.)
\end{itemize}

\pparagraph{Incommensurability.}
 Fundamentally, as in~\cite{CGS2005, CohenGS}, our construction rests on the incommeasurability of two distinct exponential growth rates (one arising as the growth of $G$, the other  arbitrarily taken to be 2 or 3.) The illustration below demontrates a similar phenomenon in the hyperbolic plane (drawn in the ``horocyclic model'': vertical distances are accurate and horizontal ones scale exponentially with height; horozontal lines are horocycles).

\vspace{\baselineskip}
\centerline{\includegraphics[scale=1.2]{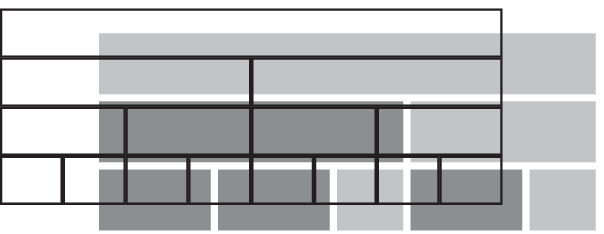}}\label{page:orbitgraphfigure}

 Two patterns of ``rectangles'' are shown, each rectangle having some predecessor above and some successors below. In the pattern drawn with dark lines, the number of rectangles doubles from row to row. In the gray pattern, light rectangles (which are all congruent) have one light and one dark rectangle as successors, and dark rectangles (which are all congruent) have one light and two dark successors. This system, asymptotically, has growth rate of $\phi^2=((1+\sqrt{5})/2)^2$ ($\phi$ is the golden ratio). The ratio of the spacing from row to row in either system is precisely fixed in relation to the other: $\log 2/\log \phi^2$. As this is not rational, the exact pattern of overlaps can never quite repeat from row to row.

 By \cite[Main Technical Lemma]{CohenGS} one may 
 produce a strongly aperiodic tileset by decorating the gray tiles with the possible combinatorial data describing how they intersect the other tiling, such as how many dark lines intersect each edge of a gray tile, and requiring these  decorations to match from tile to tile. More specifically, the sequence $\Delta_i$ consisting of the number of  horizontal dark lines meeting the $i$-th row of gray tiles could not be a periodic sequence, precisely because $\frac{\log 2}{\log(\phi^2)}$ is irrational.

We will exploit this idea in our construction. Roughly speaking, we will be using ``shortlex shellings'' to provide the underlying weakly aperiodic scaffolding (analogous to the gray tiling), on which we will place a second structure with incommeasurate growth rate, ``populated shellings''.

\subsection{Shortlex shellings.}  In Section~\ref{section:shortlexShellings}, 
we construct subshifts of finite type, much in the 
style of Coornaert and Papadopoulos~\cite[\S 3,4]{CoornaertPapadopoulosBook} and Gromov \cite[\S 7.5, 7.6, 8.4]{Gromov87}, 
which parameterize objects we call \em shortlex shellings\em\:(Definition \ref{definition:shortlexshelling}).
A shortlex shelling assigns some data to 
each element of $G$. These data impose two simultaneous, compatible structures on $G$: a decomposition into 
horospherical layers (i.e., layers which are locally modeled on spheres in $G$), and a spanning forest locally 
modeled on the tree of shortlex geodesics. 
We informally describe this here:

Given an arbitrary finitely generated group, 
with an ordered finite generating set,
every group element \(g\) is represented by a unique word 
that is, first, a shortest representative of \(g\) (that is, a geodesic) and second, earliest in the lexicographic
ordering among all such geodesics (that is, a shortlex geodesic).
In hyperbolic groups, the shortlex geodesics form a regular language, accepted by a ``shortlex finite state automaton''.

We  define a model shelling, $X_0$: to each group element $g\in G$ we associate the integer 
$h_0(g)=\dist(g,1_G)$, the  state $\state_0(g)$ of $g$ in the shortlex FSA, and, for $g\neq 1_G$, 
$P_0(g)$, the unique element of \(G\) that precedes $g$ on its shortlex geodesic from $1_G$.
A shortlex shelling is a function $X=(h,\state,P)$ modeled on   
\((h_{0},\state_{0},P_{0})\) away from the identity (up to an additive constant for \(h\)).  This means that on every finite subset $F\subset G$, the restricition of $X$ to $F$ behaves the same as the restriction of $X_0$ to some translate of $F$ which doesn't contain the identity, up to adding some constant integer, depending on $F$, to $h$.

A shortlex shelling $X=(h,\state,P)$ is encoded by 
``local data'' \(\local X = (\deriv h,\state,\dispP)\), 
a function from $G$ to a fixed finite set, where
(for \(g \in G\) and \(a \in \genset\), a finite generating set for \(G\))
\(\deriv h(g):\genset \to \{-1,0,1\}\) is the derivative of $h$, defined as
\[
\deriv h(g)(a) := h(ga) - h(g)
\]
and \(\dispP: G \to \genset\)
is defined by taking $\dispP(g)$ to be the generator \(a\) that takes us from \(g\) to \(P(g)\), that is, \(P(g) = ga\). We refer to level sets of $h$ as horospheres (of $X$).

We will construct local rules that are satisfied exactly by the local data of shortlex shellings,
showing that the set $\{\local X:X\:\text{is a shortlex shelling}\}$ forms a nonempty SFT  (Proposition \ref{proposition:LocalXisSFT}).  This SFT factors onto \(\partial G\),
the Gromov boundary of \(G\), as do the subshifts suggested by 
Gromov~\cite[\S 7.5, 7.6, 8.4]{Gromov87} and those constructed by
Coornaert and Papadopoulos~\cite[\S 3, 4]{CoornaertPapadopoulosBook}.
{
Since points of $\partial{G}$ have virtually cyclic stabilizers, it follows that these subshifts are all weakly aperiodic,
more specifically, the
stabilizer of any configuration in any of these subshifts is virtually 
cyclic and hence has infinite index (recall that $G$ is one ended).}  
However, for any hyperbolic group,
each of these subshifts admits an element with infinite cyclic stabilizer and is \em not \em  
strongly aperiodic.

To that end we construct populated shellings.

\subsection{Populated shellings.} In Section~\ref{Section:PopulatedShellings}, we begin by fixing $q\in\{2,3\}$ such that no power of 
$q$ is a power of the growth rate $\lambda$ of our shortlex machine; we say that \(q\) and \(\lambda\) are \em incommensurable\em.  
We are going to define ``populated shellings'', which decorate shortlex shellings with some extra data in order to kill any 
infinite cyclic periods, obtaining strong aperiodicity.  
In particular, a populated shelling of $G$ consists of the following data.

\begin{itemize}
\item a shortlex shelling $X=(h,\state,P)$
\item a ``population'' function $\popfunc:G\To \{0,\ldots,n\}$, for fixed $ n \in \Nn$
\item a ``population growth'' function $\density$ constant on horospheres of $X$
\item and a  ``parent-child matching'' function $m$
\end{itemize}

We further require that this data satisfies the following local rules.
We think of vertices of $G$ as being villages, some of which are inhabited by 
people---$\popfunc(g)$ tells us the number of people living in $g$. Each person has some children who live nearby 
(at a bounded distance) in the next horosphere of $X$, and $m$ describes this relationship, matching each child to its 
parent. Each person has exactly one parent, and a person living at some $g\in G$ has $q^{\Delta(g)}$ children.

We suggest this in the drawing below, with each parent living in a village in the lower horosphere having three children nearby in the next successive horosphere.

\centerline{\includegraphics[scale=.25]{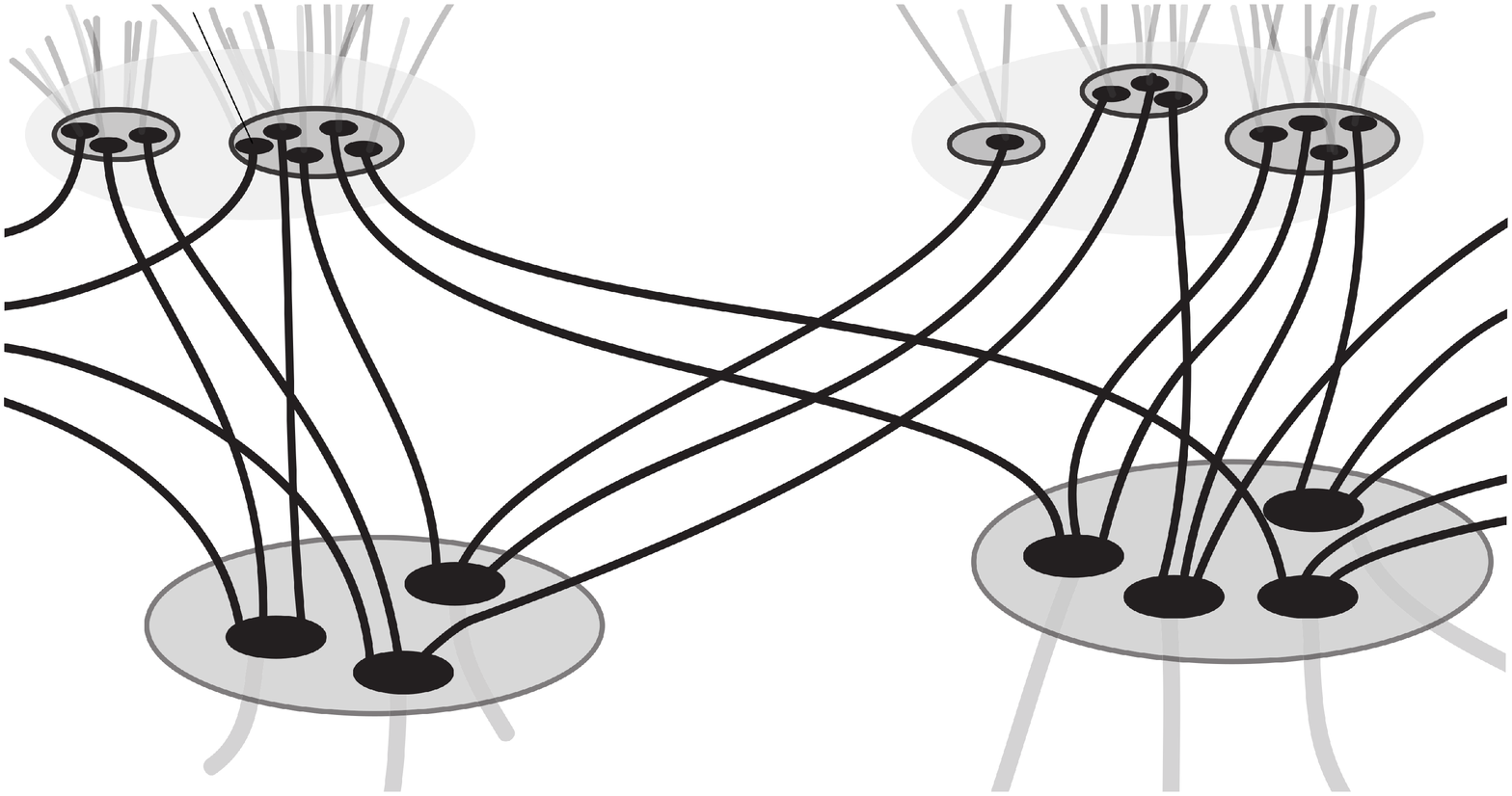}}

For a populated shelling $Y$, all of this information may be encoded by a function $\local Y$, called the ``local data'' of the populated shelling, from $G$ to a fixed finite set. Furthermore, Proposition \ref{lemma:SFTP} says that there exists a certain set of local rules such that the functions which satisfy these rules are exactly the local data of populated shellings. In other words, the set of all possible local data of populated shellings forms a SFT.

To prove our theorem, we  show
\begin{itemize}
\item that populated shellings exist (Proposition~\ref{Cor:PopShellingsExist}),
\item and that their local data cannot have infinite order periods (Proposition~\ref{Prop:Aperiodic})
\end{itemize}

\pparagraph{Infinite order periods.} We use the values of
\(\Delta\) to show that no populated shelling admits an infinite order period.   
Recall that \(\Delta\) was defined on group elements and required to be constant along on horospheres.  As the horospheres 
naturally form a sequence, the values of \(\Delta\) inherit a structure of a sequence { 
\((\Delta_i)\)}.  We will show that this sequence is not periodic,
and that this implies that there are no infinite order periods (this idea dates back to~\cite{Kari96}).
In Section~\ref{subsection:aperiodicity} we will see that for certain finite sets $S$, 
the cardinality of $P^{-n}(S)$ must grow as $\lambda^n$. On the other hand,
using the fact that quasi geodesics stay close to geodesics in a hyperbolic group,
Lemma \ref{LemmaDesc} will show that

\begin{itemize}
\item a sufficiently large finite set $S$, contained in a single horosphere, 
contains a person all of whose descendants live in 
$P^{-n}(S)$; therefore the population of $P^{-n}(S)$ grows at least 
as fast as $q^{\sum\Delta}$ (the number of descendants of that person).\footnote{
By \(\sum\Delta\) we mean the sum of the values of {\((\Delta_i)\)} along the \(n\) horospheres starting with the horosphere
containing \(S\).}
\item all descendants of people in $S$ live in $P^{-n}(S')$ for some finite set $S' \supset S$, which will imply 
that the population of $P^{-n}(S)$ grows at most as fast as {$K q^{\sum\Delta}$ 
(the number of descendants of the population of \(S'\); here \(K>0\) is the total population of \(S'\)).}
\end{itemize}

From this, it easily follows that $\frac{1}{n}\sum\Delta\log(q)\To\log(\lambda)$, which implies that the sequence  {$(\Delta_i)$} cannot be periodic by our incommensurability hypothesis.  Lemma \ref{lemma:shortlexPeriodic} implies that {$(\Delta_i)$}  would be periodic if the populated shelling $Y$ had a period of infinite order, so we conclude (Proposition \ref{Prop:Aperiodic}) that $Y$ has no infinite order period.

\pparagraph{Existence.} In Subsection~\ref{subsection:existence} we show that populated shellings exist (Proposition \ref{Cor:PopShellingsExist}), using the following strategy.
\begin{itemize}
\item We construct a sequence $(\nu_i,\density_i)$ such that {each $\nu_i\in[A,qA]$ for an arbitrary fixed  $A$,  and $\density_i \in\{\lfloor{\log_q\lambda}\rfloor, \lceil{\log_q\lambda}\rceil\}$, satisfying
$ q^{\density_i} \nu_i =\lambda \nu_{i+1} $. } In the figure on page~\pageref{page:orbitgraphfigure}, this  $\density_i$ is analogous to the number of dark horizontal lines meeting the $i$-the row of the gray tiling, while $\nu_i$ is analogous to the average frequency of dark-outlined tiles meeting each gray tile in the $i$-th row of the gray tiling.
\item We show that, given such a sequence $\nu_i$, it is possible to populate horospheres so that the $i$-th horosphere has population density $\nu_i$. In particular, the sum of $\popfunc$ over any finite set in a horosphere is equal to $\nu_i$ times $\mu(S)$ up to error bounded by $2\mu(\partial S)$ ($\mu$ will be defined defined momentarily).
\item We use the Hall Marriage trick to show that when a density sequence is realized by a population function in this way, one may find a suitable parent child matching.
\end{itemize}

\subsection{Technical tools.} Let $X=(h,\state,P)$ be a shortlex shelling.

\pparagraph{Measure.}
 In order to regularize the growth of sets under $P^{-1}$ we describe a nonnegative function $\mu$ defined on states of the shortlex machine with the following properties:
\begin{itemize}
\item
$\mu$ of a state $a$ times $\lambda$ is equal to the sum of $\mu(b)$ over the states $b$ which may follow $a$ in the 
shortlex machine, so that for any $w\in G$ we have $$\sum_{v:P(v)=w} \mu\circ\state(v) =\lambda \left( \mu\circ\state(w)\right)$$
\item The set of vertices on which $\mu\circ\state$ is positive is dense.
\end{itemize}

In Section~\ref{subsection:GrowthofFSA}, similarly to~\cite{dfw},  we  produce a left eigenvector of the transition matrix of the shortlex machine, with eigenvalue $\lambda$, whose support consists of states 
with ``maximal growth''---that is, states whose number of $n$-th successors grows at the same rate as the group itself.
 In 
Section~\ref{Section:InvariantMeasure}, we confirm that such states are dense in every shortlex shelling.

\centerline{\includegraphics[width=.9\textwidth]{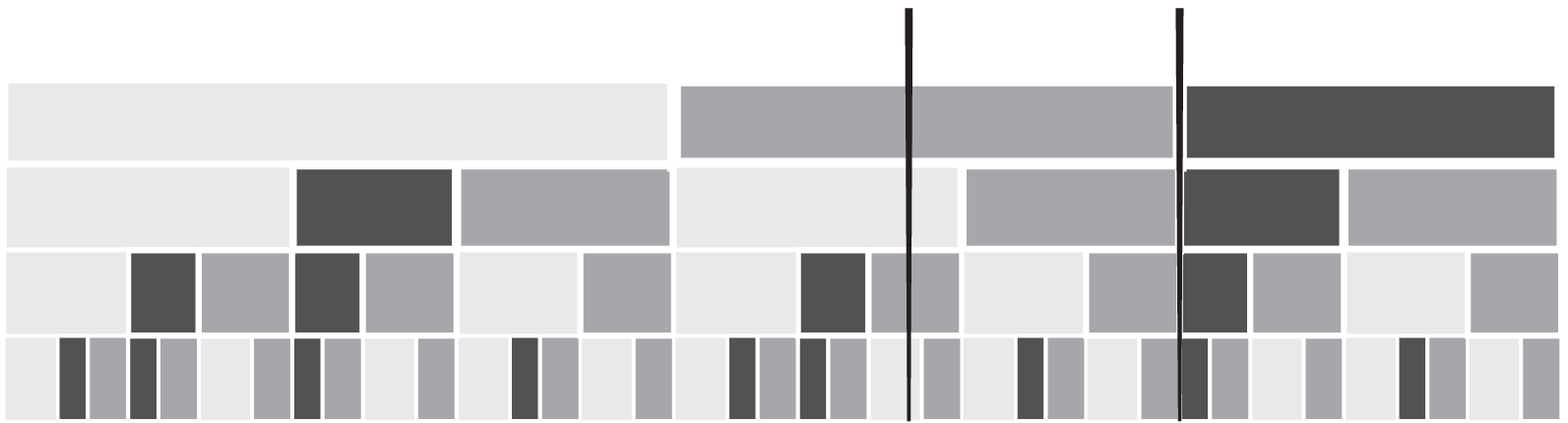}}
\label{pageWithGrowthPicture}

 In the figure above, we see a similar phenomenon in the (horocyclic model of the) hyperbolic plane: There are three types of ``rectangular'' tiles, representing three states, say  $\tta, \ttc$ and $\ttb$; the ways these tiles may fit together one above some others, represent the FSA  transitions $\tta\mapsto \tta, \ttb, \ttc$;  $\ttc\mapsto \tta, \ttc$, and  $\ttb\mapsto \ttb, \ttc$. The widths of each rectangle  are precisely in proportion to the left eigenvector of the corresponding  transition matrix.

(Two possible infinite cyclic symmetries are marked, in the middle shifting by 2 rows, and at right shifting by 1 row.  It is not possible for both to continue one more layer up. The relative heights of these tiles depends on the metric of the model on the page, but is fixed relative to  any tiling based  on another  FSA, as the ratio of the logs of their corresponding eigenvalues.)

\pparagraph{Divergence graphs.} In Section~\ref{section:divergenceMetric},  we define a graph structure on a horosphere $H=h^{-1}(n)$ known as the ``divergence graph'', 
where vertices are points $v\in H$ such that $\mu\circ\state(v)$ is positive, and two such vertices $v,w$ are connected by an edge exactly when their successor sets $P^{-n}\{v\},P^{-n}\{w\}$ remain at a bounded distance as $n\To\infty$.
 These divergence graphs have two  advantages:
 
 First, they  behave nicely under the successor map $P^{-1}$:  any pair of  vertices connected by an edge  will have some pair of successors that are also 
  connected by an edge.  In other words, every edge has one or more successor edges and either a vertex or an edge as itspredecessor, as  indicated in the  figure below, with a  larger, paler predecessor divergence graph in the background, on one horosphere, and a smaller, darker successor on the next horosphere.  
 
 \centerline{\includegraphics[scale=.18]{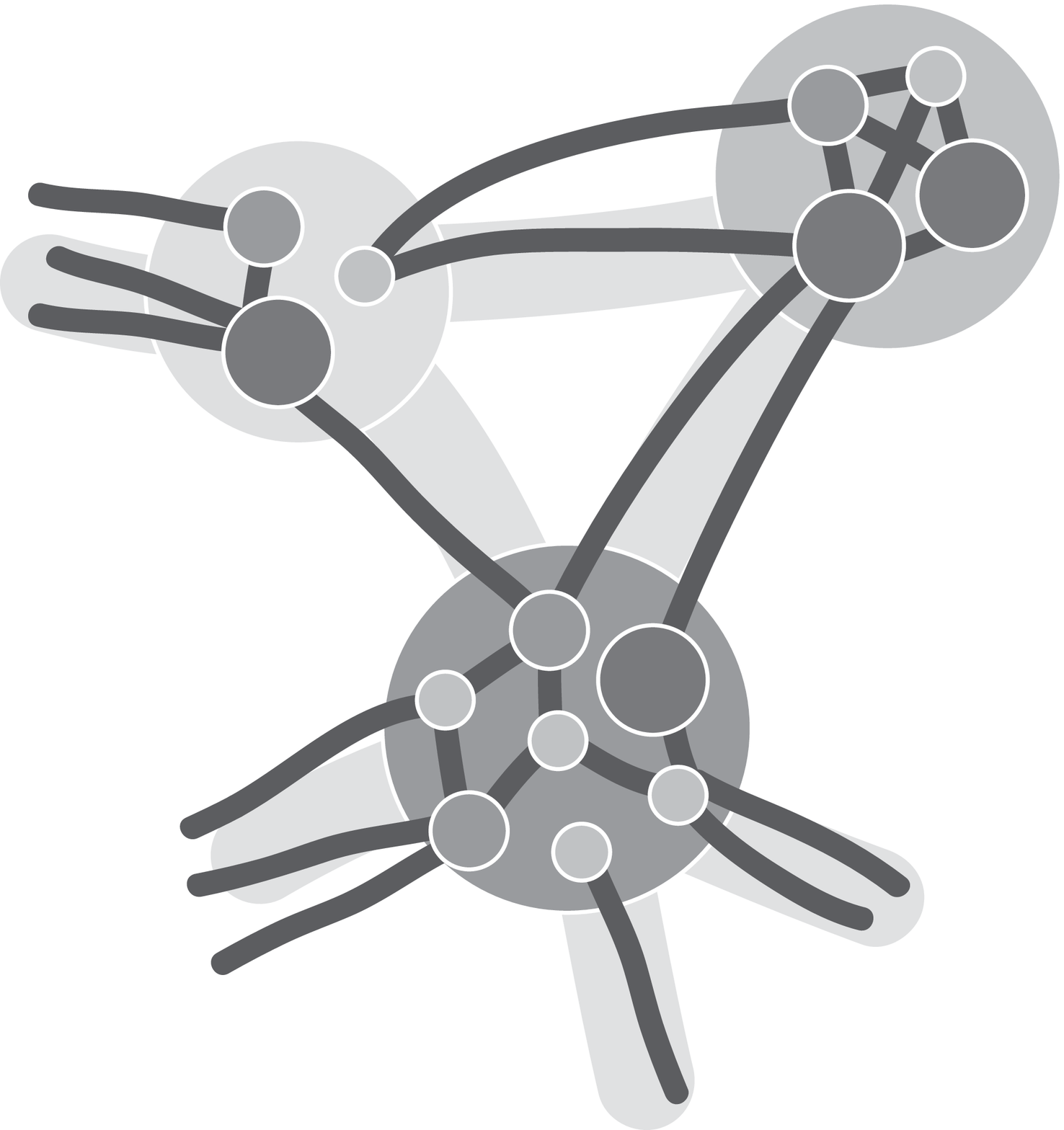}}

 Second, exactly when a hyperbolic group  is one-ended,  its divergence graphs are connected (Lemma~\ref{LemmaDivGraphConnected}), as we show using  
the
cutpoint conjecture (proved in~\cite{swarup}). This is necessary in our construction, in order to ensure that there are local rules which force the growth rate $\density$ to be constant on each particular horosphere.

\pparagraph{Translation-like actions.}  In  order to distribute the density of villagers about a horosphere, we shall use a translation-like action of 
$\Zz$ on the divergence---that is, we cover the vertices of the divergence graph with disjoint ``paths'' or injected images of $\Zz$. 

\centerline{\includegraphics[scale=.4]{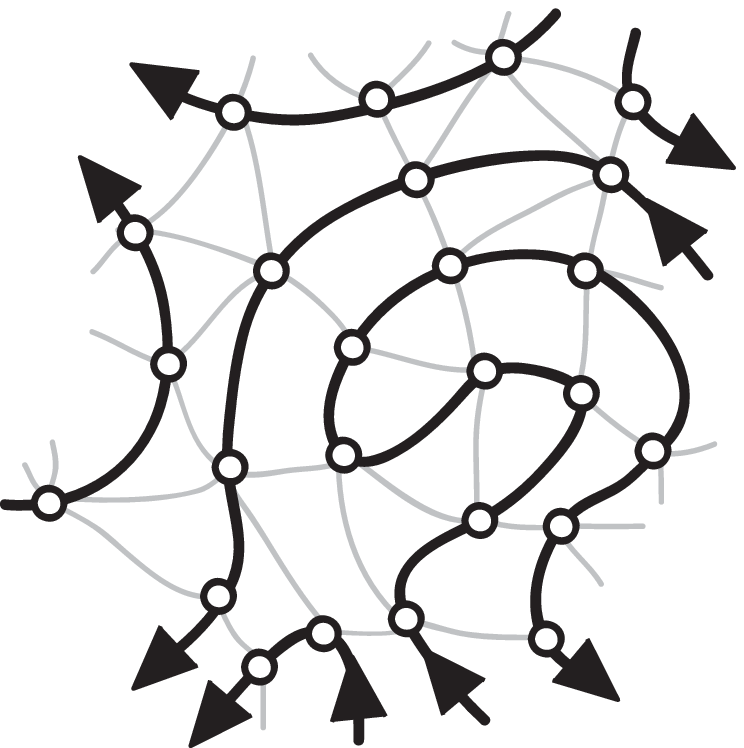}}

A theorem of Seward shows that this may be done on any one or two-ended connected graph with bounded degree, and in Section~\ref{section:TranslationlikeZactions}, we generalize this to any infinite connected graph with bounded degree.

\section{Set up}
\label{section:Definitions}

In this section, we establish our conventions and notation, and give foundational material for our construction. Subsection \ref{subsectionSFTs} recalls the definition of an SFT and explains why, in proving our main theorem, it is enough to give an SFT without infinite order periods. Subsection \ref{subsectionHyperbolicGroups} gives the definition of hyperbolic groups and their boundaries, as well as several lemmas describing their geometry which will be used throughout the sequel. Subsection \ref{subsection:GrowthofFSA} defines the shortlex automaton for a hyperbolic group and proves the important Proposition \ref{proposition:eigenvector}, which says that we may weight each state of the shortlex FSA so that states of maximal growth have positive weight and the sum of the weights of the successors of any state $a$ is equal to the growth rate of the group times the weight of $a$. Subsection \ref{subsection:horofunctions} defines the derivative of a $1$-Lipschitz function on a finitely generated group, and describes what we mean by ``horofunction''.

We take $\Nn:=\{1,2,3,\ldots\}$.  We denote the number of elements of a finite set $A$ by $\#A$.
We denote sequences as \((a_{n})_{n \in \Nn}\) (we sometimes write \((a_{n})\)).  
The notation $[a..b]$ denotes  the interval between $a$ and $b$ in $\Zz$, that is,
$$
[a..b] := \{n \in \Zz\ |\ a \leq n \leq b\}
$$
For infinite intervals, we write $\Zz_{\geq a}$, or $\Zz_{\leq b}$, or $\Zz$. 
For sums of  values of some function, say $f$, over 
set some set $R$, we write
$f_R := \sum_{x\in R} f(x)$. We may also write $f_{m..n}:=\sum_{k=m}^n f(k)$.

We work exclusively in the discrete setting:
A graph is a pair $(V(\Gamma),E(\Gamma))$.  The edges induce a metric on the 
vertices of a connected graph by setting $\dist(u,v) = 1$ whenever $u \neq v$ are 
vertices connected by an edge. 
  A 
\em geodesic \em is a (globally) metric preserving map \(\gamma:I \to X\), where \(I\) is an 
interval and \((X,\dist)\) is a metric space; that is, for any \(t_{1},t_{2} \in I\) we have 
\(\dist(\gamma(t_{1}),\gamma(t_{2})) = |t_{1} - t_{2}|\).

This paper is concerned with a fixed finitely generated group \(G\) with identity $1_G$ and a fixed finite generating
set \(\genset = \genset^{-1}\). 

As is customary we denote the set of finite words in letters of
\(\genset\) by \(\genset^{*}\) (this includes the empty word), and identify a word in \(\genset^{*}\)
with the corresponding product \(g \in G\) and say that \(w\) \em represents \em \(g\).  
Since \(\genset\) generates \(G\), this defines a map from \(\genset^{*}\)
onto \(G\).  We denote the length of a word \(w \in \genset^{*}\) by \(l(w)\) and for \(g \in G\) we set
\[
|g| := \min \{l(w) \ |\ w \in \genset^{*}, w \text{ represents } g\}
\]

This induces a distance function on \(G\) called the \em word  metric \em 
by setting
\[
\dist(g_{1},g_{2}) = |g_{1}^{-1} g_{2}|
\]
(It is well known, and easy to see, that \(d\) is indeed a metric turning \(G\) into a geodesic space --- 
that is, for any \(g_{1},g_{2} \in G\) there exist a geodesic 
\(\gamma:[0..\dist(g_{1},g_{2})] \to G\) with \(\gamma(0) = g_{1}\) and 
\(\gamma(\dist(g_{1},g_{2})) = g_{2}\).)
We denote balls as $B(n,g):=\{h\in G\ \mid \ \dist(h,g)\leq n\}$.

Multiplication defines an action of the group on itself on the left by isometries:
\[
\dist(g g_{1},g g_{2}) 
= |(g g_{1})^{-1} g g_{2}|
= |g_{1}^{-1} g_{2}|
= \dist(g_{1},g_{2}) 
\]

We say that \(A \subset G\) is \em connected \em if there exists a path connecting
any \(g_{1},g_{2} \in A\), by which we mean that there exists \(\gamma:[a..b] \to A\) 
(for some \(a,b \in \Nn\)) so that for any \(t \in [a..b-1]\) we have that
\(\dist(\gamma(t),\gamma(t+1)) = 1\).

\begin{lem}[Discrete Arzela-Ascoli]
\label{lemma:arzela}
Let $(\gamma_n:\infint\To G)_{n\in\Nn}$ be a sequence of paths in $G$.  
If for each $t\in\infint$,  $(\gamma_n(t))_{n\in\Nn}$ is finite, then  $(\gamma_{n})$ 
subconverges pointwise to some $\gamma:\infint\To G$.

If the $\gamma_n$ are geodesics, then so is $\gamma$. In particular, any sequence of geodesic rays based at the same point subconverges to a geodesic ray.
\end{lem}
\begin{proof}
The first part is obvious.

If the $\gamma_n$ are geodesics, then for any interval $[a..b]\subset \Zz$ there is some $\gamma_{n_j}$ such that $\gamma|_{[a..b]}$ agrees with $\gamma_{n_j}|_{[a..b]}$. Consequently, $b-a=\dist(\gamma(a),\gamma(b))$. It follows that $\gamma$ is a geodesic.

Finally, if $\gamma_n$ is a sequence of geodesic rays with $\gamma_n(0)=\gamma_1(0)$ for all $n\in\Nn$, then for any $t\in\infint$, $\gamma_n(t)$ is an element of the $t$-ball around $\gamma_1(0)$, and hence can assume only finitely many values. It follows that $\gamma_n$ subconverges pointwise (to a geodesic ray).
\end{proof}

\subsection{Subshifts of finite type}\label{subsectionSFTs}

We give several standard definitions:

\begin{definition}Let $G$ be a  group, and $A$ some finite set equipped with the 
discrete topology. The \em full shift \em on $G$ is 
$A^G:=\{\omega:G\To A\}$ with the product topology and the right $G$-action given by 
$(\omega\cdot g)(h):=\omega(gh)$. By Tychonoff, $A^G$ is compact. 

A {\em cylinder set} in $\Aa^G$ is a set of the form $ \prod_{g\in G} U_g$, with each $U_g\subseteq \Aa$ and for all but finitely many $g\in G$, $U_g=\Aa$.  
A {\em clopen} set is the finite union of cylinder sets.

A subset $\sft$ of $A^G$ is said to be a {\em subshift} if it is closed and invariant
under the right $G$ action.  A subshift $\sft$ is called a {\em subshift of finite type} 
(an {\em SFT}) if there exists clopen $Z_1,\ldots,Z_n$ such that $\sft=\bigcap_{g\in G;i=1,\ldots,n}Z_i\cdot g$.
We think of the $Z_i$ that define $\sft$ as giving us ``local rules'' which determine membership in $\sft$: to determine whether $\omega\in A^G$ is a configuration of $\sft$, we must see whether $\omega\cdot g$ is in $Z_i$ for all $g\in G$ and $i=1,\ldots,n$. In other words, we must check that $\omega$ takes on a prescribed form near every point in $G$.

We say that an SFT $\sft$ is {\em strongly aperiodic} if it is nonempty and for any configuration $\omega \in \sft$ we have that $\text{Stab}_G\omega=\{1_G\}$, where $\text{Stab}_G\omega$ is the stabilizer of $\omega$. 
\end{definition}

We note that our definition of strong aperiodicity is strict; some authors allow configurations in 
$\sft$ to have torsion stabilizers.
In the next proposition we observe that any group with only finitely many conjugacy
classes of torsion elements admits an SFT with no torsion stabilizers
(infinite order stabilizers may exist).  
It is well known that hyperbolic groups satisfy this condition
(see, for example,~\cite[Theorem III.$\Gamma$.3.2]{BridsonHaefliger99}).

That finite groups admit   strongly aperiodic subshifts of finite type is  trivial, but we include this within the following proposition for  efficiency.

\begin{prop} 
\label{prop:NoTorsion}
Any finitely generated group with finitely many conjugacy classes of torsion elements
admits a non-empty SFT $\sft$ such that for all $\omega\in \sft$,  
$\text{Stab}_G\omega$ has no torsion elements.

It follows that  any finite group admits a strongly aperiodic SFT. 
\end{prop} 

\begin{proof}

Let $g_1,\ldots, g_n$ be representatives of the conjugacy classes of the torsion elements in $G$.  
Let  $N := \max g_i$ and  $B:=B(N,1_G)$. We define our SFT $\sft\subset B^G$  to be such that 
for any $\omega \in \sft$ and any $g,g' \in G$, if $\dist(g,g') \leq N$ then
$\omega(g) \neq \omega(g')$.
 
By induction on the elements of $G$, $\Omega$ is non-empty: suppose we have assigned elements of $B$ 
to some subset $H$ of $G$. Let $g\in G\setminus H$. This $g$ is within $N$ of at most $\#B-1$ elements 
of $H$, and so can be assigned some element of $B$ distinct from any of those assigned to elements of 
$H$.  This process defines an element $\omega \in \Omega$, showing that $\Omega$ is not empty.
 
Let $h$ be a torsion element of $G$, with $h=cg_i c^{-1}$ for some $c$ and {representative torsion element} $g_i$. Then $\dist(c,hc)=
\dist(c,cg_ic^{-1}c)=\dist(c,cg_i)=\dist(1_G,g_i)=|g_i|\leq N$.

Thus, for $\omega\in \sft$, $\omega(c)\neq \omega(hc)$ and so $\omega\cdot h \neq \omega$. The proposition follows. \end{proof}

Our main result would give an SFT in which no configuration in stabilized by an element of 
infinite order.  The next proposition shows that we can combine it with an SFT
as constructed above to obtain a strongly aperiodic SFT:

\begin{prop} 
\label{prop:CombiningSFTs}
If group $G$ admits a non-empty SFT $\sft_1$ such that for all $\omega_1\in \sft_1$,  
$\text{Stab}_G\omega_1$ has no torsion elements, and 
$G$ admits a non-empty SFT $\sft_2$ such that
for all 
$\omega_2\in \sft_2$, $\text{Stab}_G\omega_2$ has no infinite order elements, 
then $G$ admits a strongly aperiodic SFT.
\end{prop}

\begin{proof}
Consider $\sft = \sft_1 \times \sft_2$ with the diagonal $G$-action.  
Suppose $\omega =(\omega_1,\omega_2) \in \sft$ is invariant under $g \in G$.
Then both $\omega_1$ and $\omega_2$ are invariant under $g \in G$, 
showing that $g$ is neither torsion nor has infinite order, hence $g$ is trivial.
\end{proof}

\subsection{Hyperbolic groups}\label{subsectionHyperbolicGroups}

Let $G$ be a group generated by a finite set $\genset$. We define consider $G$ with the word metric with respect to $\genset$.

\vspace{\baselineskip}
\centerline{\includegraphics[scale=.25]{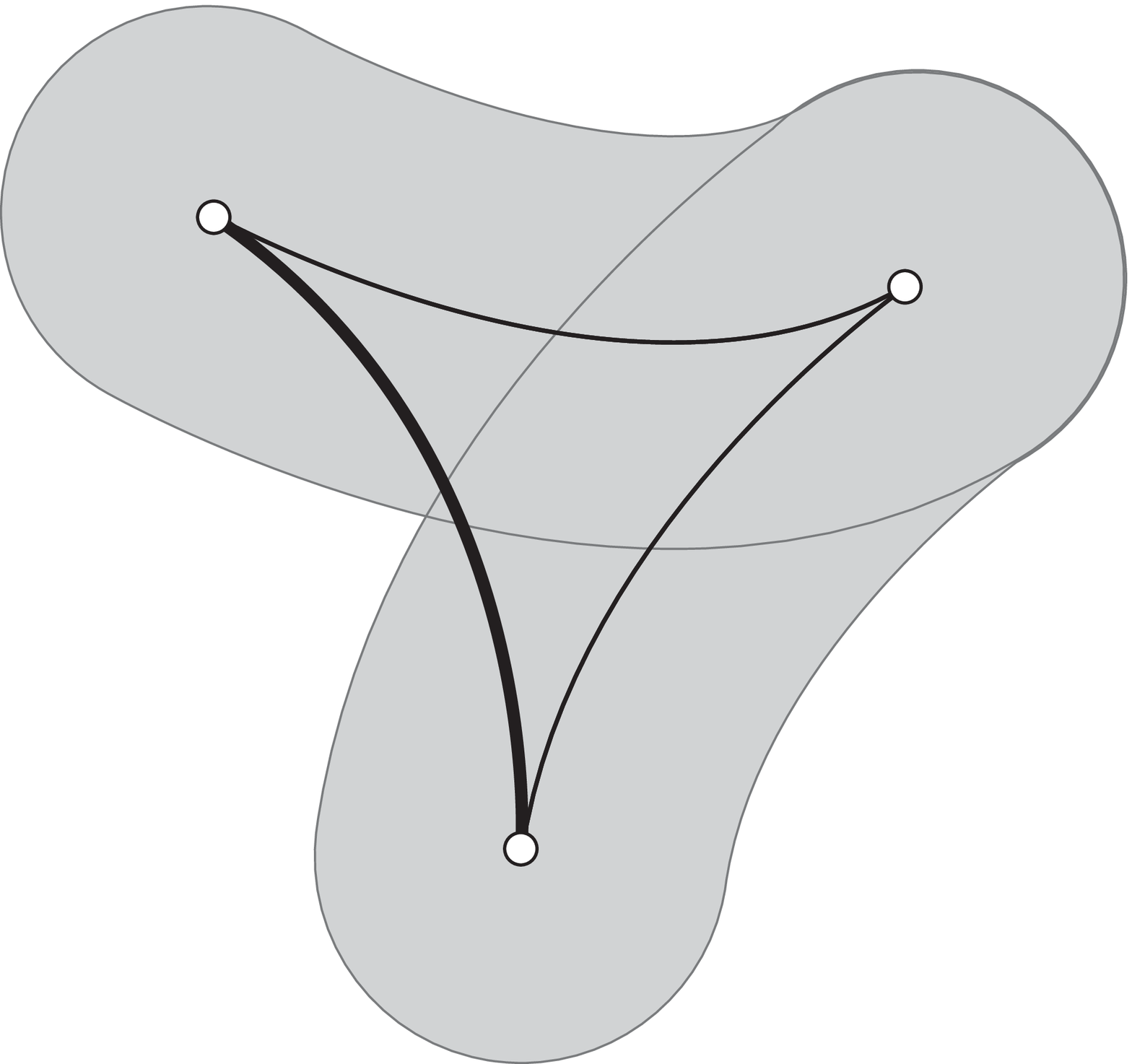}}

\begin{definition}
Let $\delta>0$. {As suggested in the above diagram}, a geodesic triangle in 
$G$
is said to be $\delta$-slim if every side is contained in the $\delta$-neighborhood of the other two side. We say that 
$G$ 
is $\delta$-hyperbolic if every geodesic triangle in 
$G$  
is $\delta$-slim. If for some $\delta$, 
$G$  
is $\delta$-hyperbolic, we call $G$ a hyperbolic group.
\end{definition}

\pparagraph{Properties of hyperbolic groups.} A hyperbolic group is finitely 
presented~\cite[Corollary 3.26]{BridsonHaefliger99} and has, at most, finitely 
many conjugacy classes of finite order elements~\cite[Theorem~3.2]{BridsonHaefliger99}.

\begin{definition}
We say that $G$ is one ended if for all compact 
$K\subset H$, 
$G\setminus K$
contains exactly one unbounded connected component.
\end{definition}

\pparagraph{Assumption.} Henceforth, we assume that $G$ is a one-ended hyperbolic group equipped with a fixed finite generating set $\genset$.

The following lemma introduces a technique which will be used frequently.
\begin{lem}[Repairing a ladder]
\label{lemma:ladder}
Let $\interval,\interval'$ be intervals containing $0$ and let $\gamma:\interval\To G$ and $\gamma':\interval'\To G$ be geodesics with $\gamma(0)=\gamma'(0)$. If $\dist(\gamma(t),\gamma'(t'))\leq k$ for some $t\in \interval,t'\in \interval'$, then $\dist(\gamma(t),\gamma'(t))\leq 2k$.
\end{lem}

\psfrag{h}[l][l]{$\gamma'(t')$}
\psfrag{i}[r][r]{$\gamma(t)$}
\psfrag{k}{$k$}
\psfrag{t}[l][l]{$\gamma'(t)$}
\psfrag{j}[c][c]{$\gamma(0)=\gamma'(0)$}
\vspace{\baselineskip}
\centerline{\includegraphics[scale=.25]{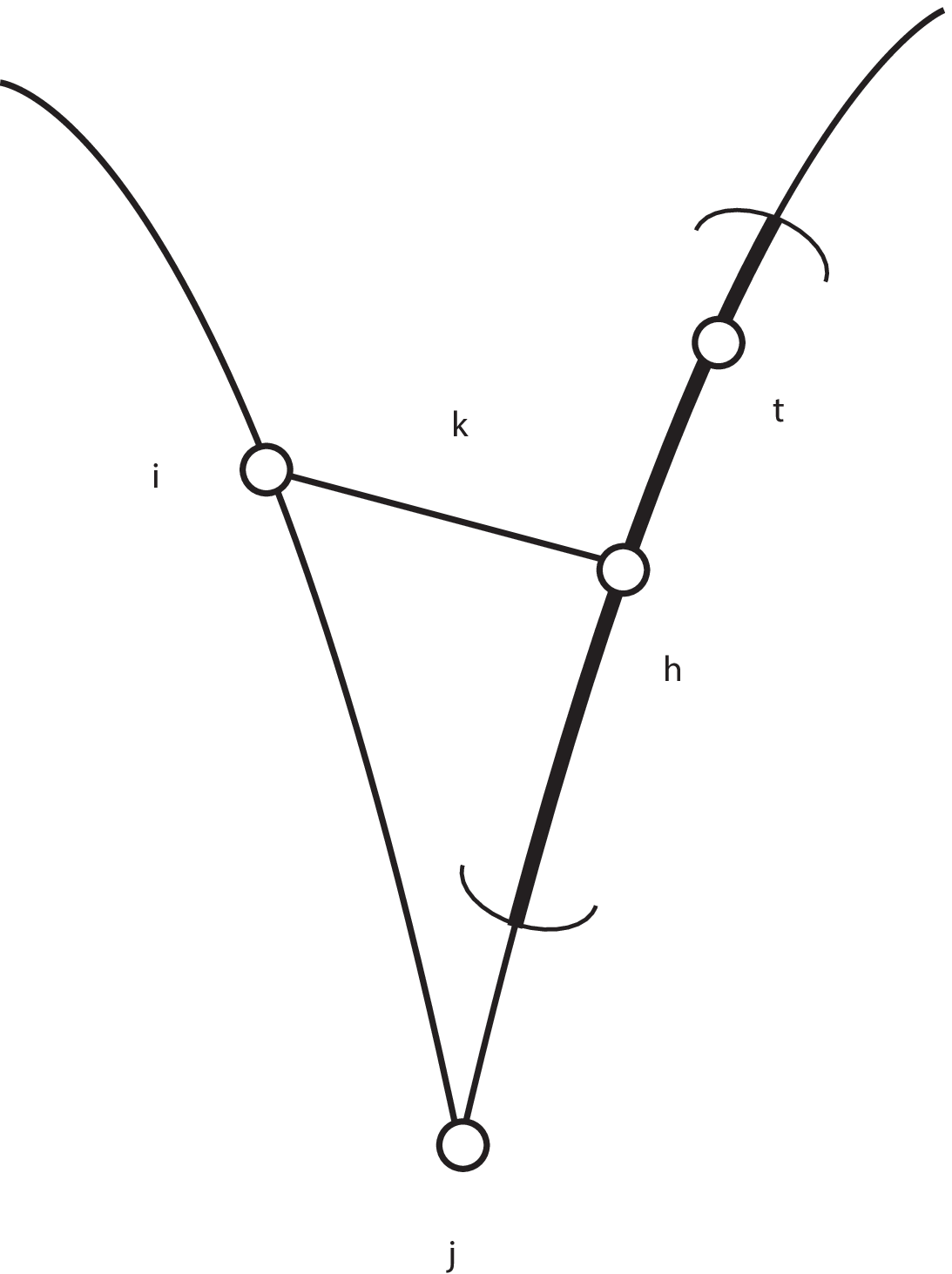}}

\begin{proof}
Because $\gamma(0)=\gamma'(0)$ we have
$$t'=\dist(\gamma'(0),\gamma'(t'))
\leq \dist(\gamma(0),\gamma(t))+\dist(\gamma(t),\gamma(t'))
\leq t+k,$$
and by symmetry, $t\leq t'+k$, so that
$$\dist(\gamma'(t'),\gamma'(t))=|t-t'|\leq k.$$
It follows that
$$\dist(\gamma(t),\gamma'(t))
\leq\dist(\gamma(t),\gamma'(t'))+\dist(\gamma'(t),\gamma'(t'))\leq 2k.$$
\end{proof}

 The next lemma gives some bounds on how long two geodesics from the same point will fellow travel.
\begin{lem}
\label{lemma:thintriangle}
Let $\interval$ and be an interval containing $0$, and let $\gamma:\interval\to G$ and $\gamma':\interval\To G$ be geodesics with $\gamma(0)=\gamma'(0)$. Suppose $t, T\in \interval$ are such that $t < T-\dist(\gamma(T),\gamma'(T))-2\delta,$ then $\dist(\gamma(t),\gamma'(t))\leq 2\delta.$
\end{lem}

\vspace{\baselineskip}
\psfrag{a}[r][r]{$\gamma(T)$}
\psfrag{c}[l][l]{$\gamma'(T)$}
\psfrag{d}[l][l]{$\gamma'(t)$}
\psfrag{e}[l][l]{$\gamma'(t')$}
\psfrag{f}[l][l]{$\gamma(0)=\gamma'(0)$}
\psfrag{g}[r][r]{$\gamma(t)$}
\psfrag{b}{$\tilga$}
\centerline{\includegraphics[scale=.25]{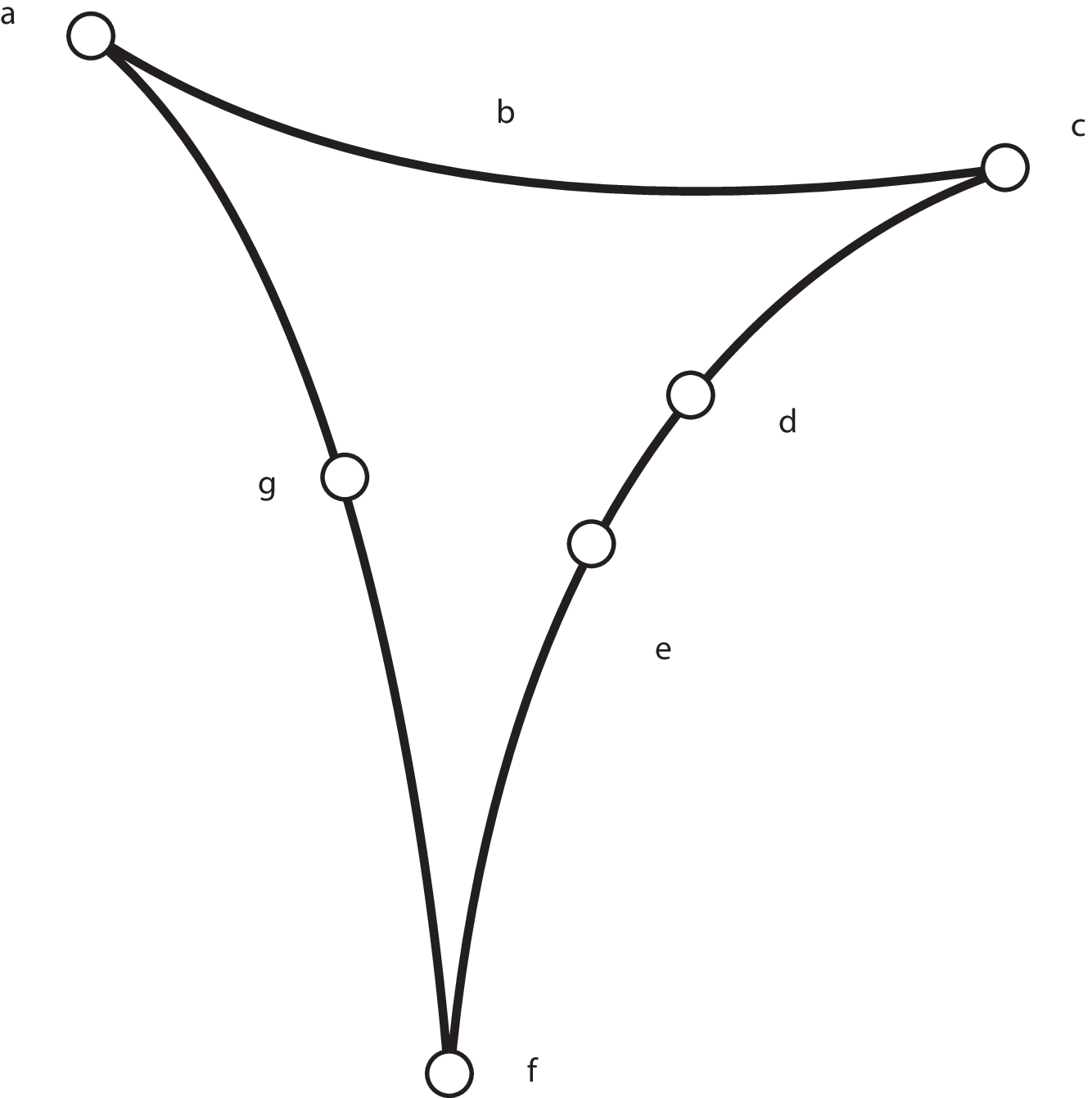}}

\begin{proof}
By slim triangles, $\gamma(t)$ is within $\delta$ of either $\gamma'$ or the geodesic $\tilga$ connecting $\gamma(T)$ to $\gamma'(T)$. In the latter case we have some $t'$ such that $\dist(\tilga(t'),\gamma(t))\leq \delta$, and thus
$$\dist(\gamma(t),\gamma(T))\leq \dist(\gamma(t),\tilga(t'))+\dist(\tilga(t'),\gamma(T))\leq \delta + \dist(\gamma(T),\gamma'(T)),$$
contradicting $\dist(\gamma(t),\gamma(T))=|T-t|>\dist(\gamma(T),\gamma'(T))+2\delta.$

Hence $\gamma(t)$ is within $\delta$ of $\gamma'$, so that there is some $t'$ such that $\dist(\gamma(t),\gamma'(t'))\leq \delta$, and we may apply Lemma \ref{lemma:ladder} to see that $\dist(\gamma(t),\gamma'(t))\leq 2\delta$.
\end{proof}

\pparagraph{Slim quads.} Consider a geodesic quad, i.e., a union of geodesic segments of the form $\overline{AB},\overline{BC},\overline{CD},\overline{DA}$. Since any diagonal of the quad is in the $\delta$-neighborhood of each pair of sides it cuts off, it is clear that each side of the quad is within a $2\delta$-neighborhood of the union of the other three. We will now see how this implies bounds on the distance between corresponding points on two geodesic segments of equal length.

\begin{lem}
\label{lemma:thinquads}
Let  $\gamma, \gamma':[0..T]\To G$ be geodesics and let 
$$k_0 = \dist(\gamma(0),\gamma'(0)),\quad\quad k_T=\dist(\gamma(T),\gamma'(T)),\quad\quad k=\max\{k_0,k_T\}.$$
For $0\leq t\leq T$, we have
$$\dist(\gamma(t),\gamma'(t))\leq 3k+4\delta.$$
If $k_0+2\delta < t <T-k_T-2\delta$, then $\dist(\gamma(t),\gamma'(t))\leq \min\{k_0,k_T\}+4\delta$.
\end{lem}

\psfrag{a}[r][r]{$\gamma(T)$}
\psfrag{b}[l][l]{$\gap(T)$}
\psfrag{c}[l][l]{$\gap(T-k_T-2\delta)$}
\psfrag{d}[l][l]{$\gap(k_0+2\delta)$}
\psfrag{e}[l][l]{$\gap(0)$}
\psfrag{f}[r][r]{$\gamma(0)$}
\psfrag{g}[r][r]{$\gamma(k_0+2\delta)$}
\psfrag{h}[r][r]{$\gamma(T-k_T-2\delta)$}

\psfrag{i}[r][r]{$\tilga$}
\psfrag{j}[l][l]{$\tilga'$}

\vspace{\baselineskip}
\centerline{\includegraphics[scale=.4]{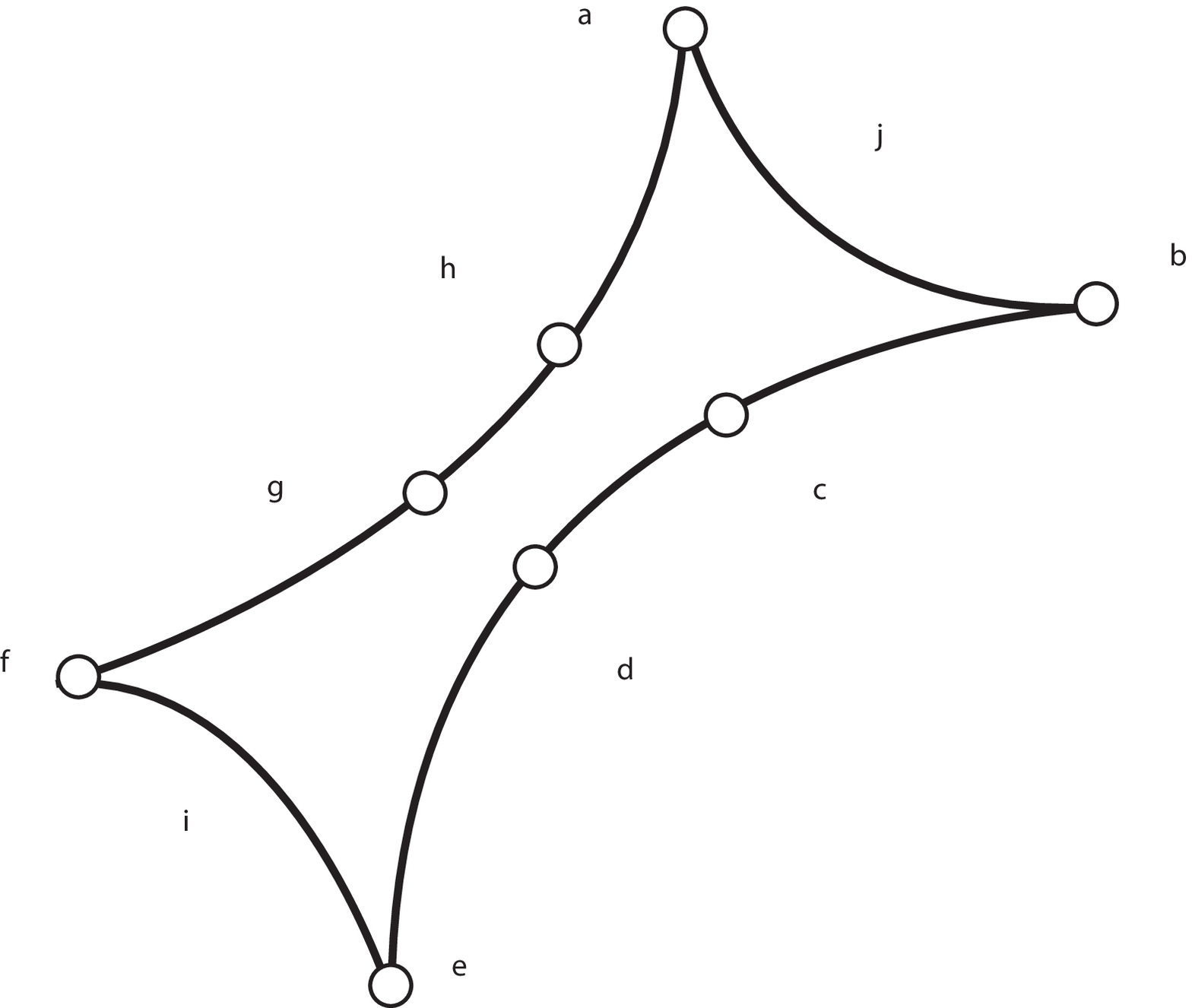}}

\begin{proof}
Let $\tilga:\tilde{\interval}\To G$ be a geodesic connecting $\gamma(0)$ to $\gap(0)$ and $\tilga':\tilde{\interval}'\To G$ a geodesic connecting $\gamma(T)$ to $\gap(T)$. Each side of the geodesic quad spanned by $\gamma,\tilga',\gap,\tilga$ is within the $2\delta$ neighborhood of the other three. In particular, $\gamma(t)$ must be within $2\delta$ of a point of $\tilga$,$\gap$ or $\tilga'$.

Suppose first that there is some $t'\in\tilde{\interval}$ such that $\dist(\gamma(t),\tilga(t'))\leq 2\delta$. By the triangle inequality, $t=\dist(\gamma(0),\gamma(t))
\leq \dist(\gamma(0),\tilga(t'))+\dist(\tilga(t'),\gamma(t))
\leq k_0+2\delta.$ It follows that
$$\dist(\gamma(t),\gamma'(t))
\leq\dist(\gamma(t),\gamma(0))+\dist(\gamma(0),\gap(0))+\dist(\gap(0),\gamma(t))
= k_0+2\delta+k_0+k_0+2\delta\leq 3k_0+4\delta.$$

The case where $\gamma(t)$ is close to some $\tilga'(t')$ is similar, so we omit the proof.

Now suppose there is some $t'\in\interval$ such that $\dist(\gamma(t),\gap(t'))\leq 2\delta$ (note that we are always in this case if $k_0+2\delta<t<T-k_T-2\delta$.) We have
$$T=\dist(\gap(0),\gap(T))
\leq\dist(\gap(0),\gap(t'))+\dist(\gap(t'),\gamma(t))
+\dist(\gamma(t),\gamma(T))+\dist(\gamma(T),\gap(T))
$$
$$\leq t'+2\delta+T-t+k_T,
$$
so that $t'\geq t-k_T-2\delta$. An entirely symmetric computation shows that $t\geq t'-k_T-2\delta$, and hence
$$|t-t'|\leq k_T+2\delta,$$
so that
$$\dist(\gamma(t),\gap(t))
\leq\dist(\gamma(t),\gap(t'))+\dist(\gap(t'),\gap(t))\leq 2\delta+|t-t'|
\leq k_T+4\delta.
$$
Reversing $\gamma$ and $\gap$, we also get the bound $\dist(\gamma(t),\gap(t))\leq k_0+4\delta$. Hence
$\dist(\gamma(t),\gamma'(t))\leq \min\{k_0,k_T\}+4\delta$
as desired.
\end{proof}

\pparagraph{Asymptotic geodesics stay close.} We will now see that the previous lemmas provide some constraints on the behavior of two geodesic rays which do not diverge from each other.

\begin{definition}
Two geodesic rays $\gamma,\gamma':\infint\To G$ are said to be asymptotic if $\dist(\gamma(t),\gamma'(t))$ is bounded---manifestly, this is an equivalence relation. We will write $[\gamma]$ for the equivalence class of $\gamma$.
\end{definition}

\begin{lem}
\label{lemma:asymptotic}
Let $\gamma,\gamma':\infint\To G$ be asymptotic geodesic rays. 
For sufficiently large $p$, there exists $q$ such that $\dist(\gamma(p),\gamma'(q))\leq 2\delta$.
Moreover, for all $t\in\infint$,
$$\dist(\gamma(t),\gamma'(t))\leq 3\dist(\gamma(0),\gap(0))+4\delta.$$ 
Finally if $\gamma(0)=\gap(0)$, then $\dist(\gamma(t),\gap(t))\leq 2\delta$ for all $t\in\infint.$
\end{lem}

\begin{proof}

Choose $k>\sup_{t\in\infint}\dist(\gamma(t),\gap(t))$.

For $p>k+2\delta$, choose $T>p+k+2$ and consider a quad with sides $\gamma|_{[0..T]}$ and $\gap|_{[0..T]}$ together with geodesic segments $\tilga$ and $\hat{\gamma}$ connecting their endpoints. Since this quad is $2\delta$-slim, we have that $\gamma(p)$ must be within $2\delta$ of one of the other three sides, and by the triangle inequality it cannot be close to $\tilga$ or $\hat{\gamma}$. It follows that for some $q$, $\dist(\gamma(p),\gamma(q))\leq 2\delta$.

Given $t$, choose $T>t+k+2\delta$. If $t\leq \dist(\gamma(0),\gap(0))+2\delta$, then we see directly that
$$\dist(\gamma(t),\gap(t))
\leq\dist(\gamma(t),\gamma(0))+\dist(\gamma(0),\gap(0))+\dist(\gap(0),\gap(T))
$$
$$\leq \dist(\gamma(0),\gap(0))+2\delta+\dist(\gamma(0),\gap(0))+\dist(\gamma(0),\gap(0))+2\delta
$$
$$=3\dist(\gamma(0),\gap(0))+4\delta$$
as desired. Otherwise, the last part of Lemma \ref{lemma:thinquads} yields the desired result.

The last part follows from Lemma \ref{lemma:ladder} and the slim triangles condition or \cite[Lemma III.H.3.3]{BridsonHaefliger99}.
\end{proof}

\pparagraph{The boundary of a hyperbolic group.} We will now define a compact space, equipped with a $G$-action, known as the boundary of $G$ (see \cite[\S III.H.3]{BridsonHaefliger99} for details.) Recall that $[\gamma]$ is the equivalence class of all rays asymptotic to $\gamma$.

\begin{definition}
Let $\partial(G)$ be the set of all equivalence classes $[\gamma]$ as $\gamma$ ranges over geodesic rays in $G$. $G$ acts on $\partial G$ via left multiplication, so that $g\cdot[\gamma]$ is given by the class of $t\mapsto g\gamma(t)$.

To define a topology on $\partial G$, fix some basepoint $p\in G$. Given $\eta_n$ a sequence of points of $\partial G$ and $\eta\in\partial G$, we say that $\eta_n$ converges to $\eta$ if $\eta_n$ can be represented by a sequence of geodesics $\gamma_n$ with $\gamma_n(0)=\gamma_1(0)$ for all $n$ and every subsequence of $\gamma_n$ subconverges pointwise to a geodesic ray representing $\eta$. We topologize $\partial G$ so that a set $K$ is closed if and only if $K$ contains the limit of every convergent sequence of points of $K$.
\end{definition}

For any choice of basepoint $p\in G$, one obtains exactly the same topology (\cite[Proposition III.H.3.7]{BridsonHaefliger99}). We sometimes write $[\gamma]$ for the element of $\partial G$ represented by a geodesic ray $\gamma$.

\begin{lem}
\label{lemma:independence}
Let $(\gamma_n),(\gamma'_n)$ be sequences of geodesic rays such that $[\gamma_n]=[\gap_n]$ for all $n$ and $\gamma_n$ converges pointwise to some geodesic $\gamma$, if $\#\{\gamma'_n(0)\}<\infty$, then $\gamma'_n$ subconverges pointwise to some $\gamma'$ asymptotic to $\gamma$.
\end{lem}

\begin{proof}
By passing to a subsequence, we may assume without loss of generality that $\gamma_n(0)$ and $\gap_n(0)$ are constant sequences. Let $k=\dist(\gamma_n(0),\gap_n(0))$. By Lemma \ref{lemma:arzela}, $\gamma'_n$ subconverges pointwise to some geodesic ray $\gap$. By Lemma \ref{lemma:asymptotic},
$$\dist(\gamma_n(t),\gap_n(t))\leq 3k+4\delta$$
for all $n$ and $t$. It follows that $\dist(\gamma(t),\gap(t))\leq
 3k+4\delta$ for all $t$, and hence $[\gamma]=[\gap]$.
\end{proof}

\subsection{Growth in a shortlex finite state automaton.}\label{subsection:GrowthofFSA}
A remarkable fact about hyperbolic groups is that the language of shortlex geodesics is regular---we recall the relevant definitions here.
For a detailed discussion see, for example,~\cite{wpig},~\cite{CalegariFujiwara}, and~\cite{dfw}

\begin{definition}
A  \em finite state automaton \em (FSA) on alphabet $\genset$ 
(where here $\genset$ is an arbitrary finite set) is a directed
graph whose edges are labeled by elements of $\genset$ (for a formal 
definition see, for example,~\cite{garey2002computers}).  
The vertices of the FSA  are called \em states\em.  Sometime we consider
FSAs that have a special state called \em start\em; in that case we only consider
finite directed paths starting at that state, and we assume that the FSA had been 
\em pruned\em, that is, states that cannot be reached from the start state have 
been removed.  Sometimes we consider FSAs without a start state, in
which case we consider all finite directed paths in the FSA.  The collection of
all words obtained by reading the edge labels of finite directed paths 
in an FSA (with or without a start state) forms
a subset of $S^*$ (the collection of  all finite words in $S$, including the empty
word); a subset of this form is called a
\em regular language\em.  
\end{definition}

\pparagraph{Notation.}  Let $\Gamma$ be an FSA with states $V(\Gamma)$. For a set of states $A\subset V(\Gamma)$, we let $\Gamma(A)$ denote the subgraph spanned by $A$ (itself an FSA).  We let $[\Gamma]$ denote the adjacency matrix (i.e., if we number the states $\{a_1,\ldots,a_n\}$, $[\Gamma]_{ij}$ denotes the number of transitions from $a_j$ to $a_i$). 
If a word $w\in S^*$ labels a valid path from a state $a$ to a state $b$, we write $a{\buildrel w\over\To} b$.  
If $a,b\in V(\Gamma)$ are such that $a{\buildrel w\over\To} b$ and $b{\buildrel w'\over\To} a$ (for some $w,w' \in S^*$), 
we say that $a\approx b$.  It is clear that $\approx$ is an equivalence relation
(note that $a\approx a$ always holds, as the path may have length zero). The equivalence classes are called components.  The Perron Frobenuis theorem asserts that:

\begin{lem}
\label{lemma:PF}
If $A\subset V(\Gamma)$ is a component and $\#A\geq 2$, then the largest modulus eigenvalue $\lambda_A$ of $[\Gamma(A)]$ is positive and has a positive left eigenvector.
\end{lem}

\pparagraph{The shortlex automaton.} 
Recall our convention that $\genset$ is a symmetric generating set for the one-ended hyperbolic group $G$.  We say that $s_1\cdots s_\ell\in\genset^*$ is a \em geodesic \em if $\ell$ is the minimal length of any word representing the same element of $G$ as $s_1\cdots s_\ell$. The collection of all geodesic words forms a regular language~\cite[Theorem 3.4.5]{wpig}. Order the elements of $\genset$---this induces a lexicographic order on $\genset^*$. A word  $s_1\cdots s_\ell$ is a \em shortlex geodesic \em 
if  it is a geodesic and no geodesic representing the same group
element precedes it in the lexicographic order.
The set of all shortlex geodesics forms  a regular language~\cite[Proposition 2.5.2]{wpig}, called \em the language of shortlex geodesics \em in $G$ (and with generators $\genset$.)

\begin{definition}
Let $\lambda:=\lim_{i\To\infty}\#B(i,1_G)^{1/i}$ be the growth rate of $G$ with respect to $\genset$~(see for example \cite{dfw}). Let $\Mm$ denote a pruned FSA for the language of shortlex geodesics in $G$, and let $\Aa$ denote the vertex set of $\Mm$.
\end{definition}

We are going to show that $\lambda$ is an eigenvalue of the transition matrix $[\Mm]$ with a left eigenvector supported on a certain set of states (later we shall see that these states are dense in $G$.) Write $\lambda_B$ for the Perron-Frobenius eigenvalue of a component $B\subset\Aa$. By \cite[Theorem 3.3, Corollary 3.7]{dfw}, $\lambda$ is equal to the maximum of the $\lambda_B$. We say that a component $B$ is {\em big} if $\lambda=\lambda_B$.

Partition $\Aa$ into sets $\Aamax\sqcup\Aabig\sqcup\Aamin$ where
\begin{itemize}
\item $\Aabig$ is the union of the big components.
\item $\Aamin$ consists of all states that cannot lead to a big component.
\item $\Aamax$ consists of everything else---i.e., states which are not in a big component but may lead to a big component.
\end{itemize}

\begin{prop}
\label{proposition:eigenvector}
There is a left eigenvector $\mu$ of $[\Mm]$ with eigenvalue $\lambda$ such that $\mu_i>0$ for $a_i\in\Aamax\cup\Aabig$ and $\mu_i=0$ for $a_i\in\Aamin$.
\end{prop}
\begin{proof}

We first construct a positive eigenvector $\mubig$ of $[\Mm(\Aabig)]$ with eigenvalue $\lambda$, then a positive eigenvector $\mu_0$ of $[\Mm(\Aamax\sqcup\Aabig)]$ with eigenvalue $\lambda$, then the desired eigenvector $\mu$.

(1) By \cite[Lemma 3.4.2]{calegari}, there is no path from one big component to another (this is a moral equivalent of the fact,
proved by Coornaert\cite{coornaertpacific}, that the growth of $G$ is precisely exponential, i.e., $\#B(n,g)=\Theta(\lambda^n)$.)
It follows that we may write $[\Mm(\Aabig)]$ as a block diagonal matrix
$$[\Mm(\Aabig)]=\begin{bmatrix}A_1& & \\ &\ldots & \\ & & A_n\end{bmatrix}$$
where each $A_i$ is $[\Mm(B)]$ for some big component $B$. Letting $\mu_i$ be the PF eigenvector for $A_i$, we have that $\mubig:=[\mu_1\:\cdots\:\mu_n]$ is a positive eigenvector for $[\Mm(\Aabig)]$ with eigenvalue $\lambda$.

\def\identity{{\mathbb{I}}}
(2) We may write
$$[\Mm(\Aamax\sqcup\Aabig)]=\begin{bmatrix}[\Mm(\Aamax)]& 0 \\B &[\Mm(\Aabig)]\end{bmatrix}$$
for some matrix $B$. Observe that $(\lambda\identity-[\Mm(\Aamax)])$ is invertible 
(where by $\identity$ we mean the identity matrix), 
with inverse given by
$$(\lambda\identity-[\Mm(\Aamax)])^{-1}=\lambda(\identity+\lambda^{-1}[\Mm(\Aamax)]+\lambda^{-2}[\Mm(\Aamax)]^2+\ldots)$$
where the series (which is nonnegative) converges because $\lambda$ is greater than any eigenvalue of $[\Mm(\Aamax)]$. We now see that
$$\mu_0:=[\mubig B(\lambda\identity-[\Mm(\Aamax)])^{-1}\quad\mubig]$$
is an eigenvector for $[\Mm(\Aamax\sqcup\Aabig)]$ by the following calculations. Write $\nu$ for $\mubig B(\lambda\identity-[\Mm(\Aamax)])^{-1}$.
$$\mubig B=\nu(\lambda\identity-[\Mm(\Aamax)])$$
Hence:
$$\nu[\Mm(\Aamax)]+\mubig B=\lambda\nu$$
which implies that $[\nu \quad \mubig]$ is a nonnegative left eigenvector of $[\Mm(\Aamax\sqcup\Aabig)]$ as desired, so we wish to show that it is positive.

 Because each state of $\Aamax$ may lead to a state of $\Aabig$, we see that for all $a_i\in\Aamax$, there is some $a_j\in\Aabig$ and $k\geq 0$ such that $[B[\Mm(\Aamax)]^k]_{ji}>0$. By the geometric series formula for $(\lambda\identity-[\Mm(\Aamax)])^{-1}$ and the fact that every $[\mubig]_j$ is positive, we thus see that every $[\mubig B(\lambda\identity-[\Mm(\Aamax)])^{-1}]_i$ is positive, and hence $\mu_0$ is positive.

(3) Finally, we may write
$$[\Mm]=\begin{bmatrix}[\Mm(\Aamax\cup\Aabig)]& 0 \\\ast &[\Mm(\Aamin)]\end{bmatrix}$$
and take $\mu:=[\mu_0\quad 0]$ as our desired eigenvector.
\end{proof}

\subsection{Horofunctions and their derivatives}
\label{subsection:horofunctions}

\begin{definition}
\rm 
Let $h:G\To\Zz$ be a $1$-Lipschitz function. The \em derivative \em 
$$\deriv h:G\To\mioi^\genset$$ 

of $h$ is the function
$$\deriv h: g\mapsto (s\mapsto h(gs)-h(g)).$$
\end{definition}

The following lemma says that two functions with the same derivative differ by a constant, as one might expect.

\begin{lem}
\label{lemma:DrivativeDeterminesUpToConstant}
Let \(h_{1}, h_{2}:\Gamma \to \mathbb{Z}\) be \(1\)-Lipschitz functions.  If \(\deriv h_{1} = \deriv h_{2}\)
then \(h_{1} - h_{2}\) is constant.
\end{lem}

\begin{proof}
\cite[Lemma 3.4]{DavidCohen17} implies that, for a Lipschitz function $h$, $h(g)-h(g')$ may be recovered from $\deriv h|_p$ where $p$ is a path connecting $g$ to $g'$. It follows that $h_1-h_2$ is constant.
\end{proof}

There are multiple (essentially but not entirely equivalent) definitions of ``horofunction'' in the literature. We will use the following:
\begin{definition}
\rm
An {onto} $1$-Lipschitz function 
$h:G\To \Zz$ is said to be a \em horofunction \em if the derivative 
$\deriv h$ is in the orbit closure of the derivative of 
the function
$$g\mapsto \dist(g,1_G)$$
Level sets of horofunctions will be referred to as \em horospheres. 

\end{definition}

For example, the horofunctions $\Zz\to\Zz$, with the integers generated by $\pm 1$, are given by $n\mapsto n+C$ and $n\mapsto -n+C$ as $C$ ranges over $\Zz$.

Note that functions in the actual orbit of $g\mapsto \dist(g,1_G)$ are not onto $\Zz$, but only some 
$\Zz_{\geq N}$, and so only limit points of an unbounded orbit of such functions can possibly be 
horofunctions. The next lemma makes this precise:

\begin{lem}
\label{lem:horofunctions}
A function $h:G \to \mathbb{Z}$ is a horofunction if and only 
if there \(g_{0} \in G\) and a sequence
\((g_{n})_{n=1}^{\infty}\)   of distinct elements of \(G\)
and such that $h$ is the pointwise limit of the sequence  \((f_{n})_{n=1}^{\infty}\) where 
$$f_{n}(g):= \dist(g,g_n)-\dist(g_n,g_0)$$
\end{lem}

\begin{proof}
Let \(h\) be a horofunction.  We will produce the points \(g_{n}\).
By definition, there exists a sequence of sets \(S_{n} \subset G\), 
\(n \in \Nn\)
satisfying:
	\begin{itemize}
	\item \(S_{n} \subset S_{n+1}\) for all \(n \in \Nn\)
	\item \(\bigcup S_{n} = G\)
	\item For each \(n\in\Nn\), there exists \(g_{n} \in G\) for which \(
	\deriv\dist(\cdot,g_{n})
	|_{S_{n}} = \deriv h|_{S_{n}}\)
	\end{itemize}
Note that these conditions imply that for any \(m \geq n\) we have that	
\(
\deriv\dist(\cdot,g_{m})
|_{S_{n}} = \deriv h|_{S_{n}}\).
By restricting to subsets of \(S_{n}\) we may assume that the graph spanned by \(S_{n}\)
is connected for each \(n\). 
 	
Since \(h\) is onto \(\mathbb{Z}\), there exists \(g_{0} \in G\) for which \(h(g_{0}) = 0\) is satisfied. Moreover, for each $r>0$, there is some $N$ such that for all $n>N$, the ball of radius $r$ centered at $g_0$ is contained within $S_n$.

Since each \(f_n\) defined in the statement of the lemma differs from \(\dist(g,g_{n})\) only by a constant, \(\deriv f_{n} = \deriv \dist(g,g_{n})\).
By the conditions above we see that \(f_{n}\) satisfies:
	\begin{itemize}
	\item \(\deriv f_{n}|_{S_{n}} = \deriv h|_{S_{n}}\)
	\item \(f_n(g_{0}) = h(g_{0})\)
	\end{itemize}
As \(S_{n}\) is connected, by Lemma~\ref{lemma:DrivativeDeterminesUpToConstant},
condition~(1) above implies that 
\(f_{n}|_{S_{n}} = h|_{S_{n}}\).  We see that
\[
\lim_{n \to \infty}f_{n} = h
\]
It remains to show that the elements may be taken as distinct. Suppose not. Then after subsequencing if necessary we may assume that 
\((g_{n})_{n=1}^{\infty}\) is a constant sequence.  In that case \(h(g) = \dist(g,g_{n})+C\) for some constant $C\in\Zz$, contradicting the 
assumption that \(h\) is onto.

The converse follows from the definitions.
\end{proof}

\begin{lem}
\label{lemma:DipSomeMore}
Let $h$ be a horofunction and $g_1,g_2 \in G$.  Suppose that
$h(g_1) = h(g_2)$.  If, for some $x \in \Nn$, we have that
$\dist(g_1,g_2) > 2x + 2\delta$, then for any geodesic
$\gamma:[0..\dist(g_1,g_2)]$ connecting $g_1$ and $g_2$ we have
$$h(\gamma(x)) \leq h(g_1) - (x - 2\delta)$$
\end{lem}

\begin{proof}
By reorienting $\gamma$ if necessary we may assume that \(\gamma(0) = g_{1}\).
By Lemma~\ref{lem:horofunctions}, there exists $g_0 \in G$ and $C \in \Nn$
so that for all $t  \in [0..\dist(g_1,g_2)]$ we have that 
$h(\gamma(t)) = \dist(g_0,\gamma(t)) - C$.  For \(i=1,2\), let \(\gamma_{i}\) be a geodesic
from \(g_{i}\) to \(g_{0}\) (so that \(\gamma_{i}(0) = g_{i}\)).  By the slim triangle inequality, 
\(\gamma(x)\) is within \(\delta\) of some point of \(\gamma_{1}\) or \(\gamma_{2}\), say \(p\).
We claim that \(p \not\in \gamma_{2}\); assume that it is.  Then \(\dist(g_{1},p) \leq x+\delta\).
Again by the triangle inequality, \(\dist(g_{0},p) \geq \dist(g_{0},g_{1})- (x+\delta)= \dist(g_{0},g_{2})- (x+\delta)\), and so,
as \(p\) is on a geodesic connecting \(g_{0}\) and \(g_{2}\), we have that \(\dist(p,g_{2}) \leq x+\delta\).
This shows that \(\dist(g_{1},g_{2}) \leq 2x + 2\delta\), contradicting our assumption.

\vspace{\baselineskip}
\psfrag{a}[r][r]{$g_1$}
\psfrag{b}[][]{\small $x$}
\psfrag{c}[][]{$\gamma(x)$}
\psfrag{d}[l][l]{$g_2$}
\psfrag{e}[l][l]{$\leq x+\delta$}
\psfrag{f}[l][l]{$\geq \dist(g_0,g_2)-(x+\delta)$}
\psfrag{g}[][]{\small $ \delta$}
\psfrag{h}[][]{}
\psfrag{i}[][]{}
\psfrag{j}[][]{}

\centerline{\includegraphics[scale=.4]{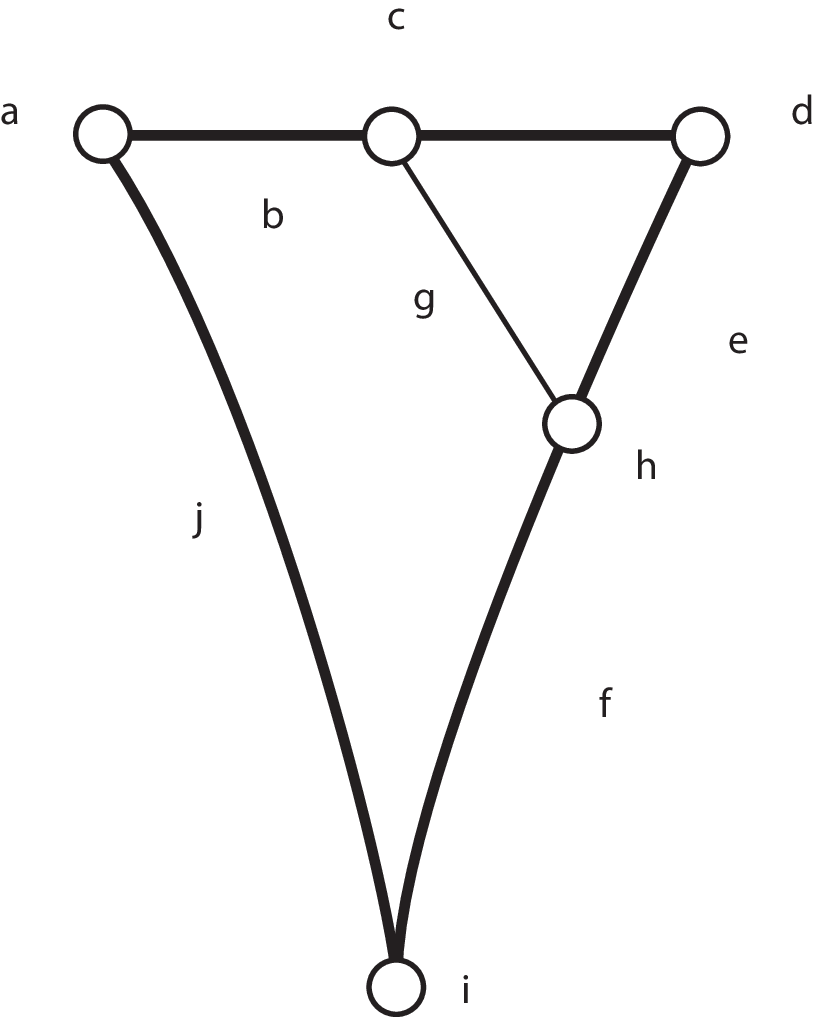}}

Therefore \(p \in \gamma_{1}\).  By Lemma~\ref{lemma:ladder} we have that 
\(\dist(\gamma(x),\gamma_{1}(x)) \leq 2 \dist(\gamma(x),p) \leq 2\delta\) and so \[
\dist(g_{0},\gamma(x)) \leq
\dist(g_{0},\gamma_{1}(x)) + \dist(\gamma_{1}(x),\gamma(x)) \leq
\dist(g_{0},g_{1}) - x + 2 \delta
\]
Thus
\[
h(\gamma(x)) = \dist(g_{0},\gamma(x)) - C
\leq  \dist(g_{0},g_{1}) - x + 2 \delta - C = 
h(g_{1}) - (x - 2\delta)
\]
\end{proof}

\section{Translation-like \(\mathbb{Z}\) actions}\label{section:TranslationlikeZactions}
A theorem of Seward asserts that every one or two-ended connected graph in which the degrees of the vertices are bounded  admits 
a translation-like $\mathbb Z$ action~\cite[Theorem~3.3]{seward}. (See below for the definition of translation-like $\mathbb Z$ action.)
Bowditch~\cite{CutPointsAndCanonicalSplittingsOfHyperbolicGroups} shows that horospheres have an arbitrarily large  number ends, and a result of Bonk and Kleiner~\cite{BonkKleiner} 
suggests that a divergence graph on a horosphere is quasi-isometric to that
horosphere. We provide a 
generalization of Seward's work, given in Proposition~\ref{translationLikeActionThm}
below,  producing a translation-like $\mathbb Z$ action on any connected infinite graph of uniformly bounded degree. This proposition plays an  important role in demonstrating the existence of our populated shellings (Lemma~\ref{densityOnHorofunctions}).

We first define:

\bigskip

\begin{definition}
Let \(\Gamma\) be a graph and \(L\) a positive integer.  
A \em translation-like \(\mathbb{Z}\) action with defect \(L\) \em on \(\Gamma\) is a bijection 
\(f:V(\Gamma) \to V(\Gamma)\) (here \(f\) is thought of as the generator of \(\mathbb{Z}\))
satisfies, for any \(x \in \Gamma\):

\begin{itemize}
\item \(\dist(x,f(x)) \leq L\)
\item \(f^{i}(x) = x\) only for \(i=0\)
\end{itemize}
\end{definition}

\begin{rmk}
\label{RemarkAboutPaths}
If \(\Gamma\) admits a translation-like \(\mathbb{Z}\)-action with defect \(L\) 
then the orbit of a vertex \(x\) is an injective map \(\mathbb{Z} \to \Gamma\)
for which the distance between the images of consecutive integers is at most \(L\) 
(in a way we can think of the orbit as a ``path'') .  It is now easy to see that  \(\Gamma\) admits a 
translation-like \(\mathbb{Z}\) action with defect \(L\) if and only if
\(\Gamma\) can be decomposed as the disjoint union of (possibly infinitely many) 
subsets (``paths'') each admitting an injective map from \(\mathbb{Z}\) satisfying this condition.
\end{rmk}

\begin{prop}\label{translationLikeActionThm}
Let \(\Gamma\) be a connected infinite graph and \(M\) a positive integer so that the degree of 
each vertex of \(\Gamma\) is at most \(M\).  Then \(\Gamma\) admits a translation-like 
\(\mathbb{Z}\) action with defect at most \(2M+1\).
\end{prop}

\begin{proof}
We apply   Zorn's Lemma.
To that end, we define a partially ordered set \(\mathcal{Z}(\Gamma)\)
(or simply \(\mathcal{Z}\), when no confusion can arise)
whose elements are pair \((X,f)\) where here
	\begin{itemize}
	\item \(X \subset V(\Gamma)\) 
	\item The graph spanned by 
	\(V(\Gamma) \setminus V(X)\) has no finite components
	\item \(f\) is a translation-like \(\mathbb{Z}\) action with defect at most \(2M+1\)
	on the graph spanned by \(X\) 
	\end{itemize}
We say that 
\((X_{1},f_{1}) \leq (X_{2},f_{2})\) if and only if:
	\begin{itemize}
	\item \(X_{1} \subset X_{2}\)
	\item \(f_{1} = f_{2}|_{X_{2}}\)
	\end{itemize}
Note that by definition a \(\mathbb{Z}\) action is given by a function on the vertices, 
so it the second condition makes sense: it says that
that \(f_{1}\) is the restriction of \(f_{2}\).  A simple way to visualize this is the following: 
by Remark~\ref{RemarkAboutPaths} above
\(f_{1}\) decomposes \(X_{1}\) into ``paths'', and similarly for \(f_{2}\).  The second
condition says that each path in \(X_{1}\) under \(f_{1}\) is a path in \(X_{2}\) under \(f_{2}\).

\medskip
Claims~\ref{first_claim} and~\ref{second_claim}
below establish that \(\mathcal{Z}\) fulfills the requirements of Zorn's lemma:

\medskip

\begin{clm}
\label{first_claim}
\(\mathcal{Z}(\Gamma)\) is not empty.
\end{clm}

\begin{proof}[Proof of Claim~\ref{first_claim}]
If \(\Gamma\) has only one end then Seward~\cite{seward} establishes the claim, and his result immediately extends to the case of two ends as well.

We assume then that \(\Gamma\) has more than one end, and hence admits a biinfinite 
geodesic, say \(\gamma\).  Let \(X_{0}\) be \(V(\gamma)\) (the vertices
of \(\gamma\)) together with the vertices of any \em bounded \em component of the graph spanned
by \(V(\Gamma) \setminus V(\gamma)\).  

We claim that \(\Gamma(X_{0})\), the graph spanned by \(X_{0}\), 
is infinite, connected,  and has at most two ends.  Since by construction  \(\Gamma(X_{0})\) is 
infinite and connected, the only worry is the possibility that it has more than two ends.

Let \(E_{1},E_{2},E_{3}\) be three ends of \(\Gamma(X_{0})\), that is, there is a finite 
set \(K \subset X_{0}\) so that for \(i=1,2,3\) we have that \(E_{i}\) is an 
infinite connected component of the graph spanned by \(X_{0} \setminus K\).

Clearly \(V(E_{i})\) contains infinitely many vertices of \(V(\gamma)\), for otherwise it 
would consist of a finite set \(F \subset V(\gamma)\) together with bounded components, each adjacent 
to at least one vertex of \(F\) (in case \(V(E_{i}) \cap V(\gamma) = \emptyset\)
we get that \(V(E_{i})\) is contained in the vertices of one bounded component).  

The finite degree of \(V(\Gamma)\) implies that there are only finitely many
bounded components adjacent to each vertex of \(F\), and we conclude that \(V(E_{i})\) is finite,
a contradiction. 

 Thus each \(E_{i}\) contains vertices of \(V(\gamma)\) that correspond to 
arbitrarily large or arbitrarily negative integers.  By renumbering if necessary we may assume that \(E_{1}\)
and \(E_{2}\) both contain vertices that correspond to arbitrarily large or arbitrarily negative integers. 

 Since \(K\)
is finite, there is a vertex \(v_{1} \in E_{1}\) and a vertex \(v_{2} \in E_{2}\), corresponding to
integers \(n_{1}\) and \(n_{2}\) so big (or so negative) that no vertex corresponding to an integer 
between the two
is in \(K\).  

Thus the segment of \(\gamma\) connecting \(v_{1}\) and \(v_{2}\) is disjoint from \(K\)
and we conclude that \(E_{1} = E_{2}\), establishing that \(\Gamma(X_{0})\) has at most two ends.

By~\cite{seward},  \(\Gamma(X_{0})\) admits a translation-like \(\mathbb{Z}\) action with defect
at most \(2d+1\), say \(f_{0}\).  Thus \((X_{0},f_{0}) \in \mathcal{Z}\), and so 
\(\mathcal{Z}\) is not empty, establishing Claim~\ref{first_claim}.
\end{proof}

\begin{clm}
\label{second_claim}
Every chain in \(\mathcal{Z}\) has an upper bound.
\end{clm}

\begin{proof}[Proof of Claim~\ref{second_claim}]
Let \(\big\{(X_{\alpha},f_{\alpha})\big\}_{\alpha \in A}\) be a chain in \(\mathcal{Z}\).  Set
\(X = \bigcup_{\alpha \in A} X_{\alpha}\)
and define a  \(\mathbb{Z}\) action \(f\) by setting \(f(x) = f_{\alpha}(x)\) for 
\(x \in X_{\alpha}\).  Since \(\big\{(X_{\alpha},f_{\alpha})\big\}_{\alpha \in A}\) is 
a chain, the definition of \(\mathcal{Z}\) shows that \(f\) is well defined.
Thus \(f\) defines a \(\mathbb{Z}\) action on \(X\) and since both conditions of Definition~1
are given pointwise, it is clear that \(f\) defines a translation-like action.  

It remains to show that every component of the graph spanned by \(V(\Gamma) \setminus X\)
is unbounded.  Suppose, for a contradiction, that  there exist a bounded component
\(\Gamma'\) of the graph spanned by \(V(\Gamma) \setminus X\).
Since the degree of the vertices of \(\Gamma\) is finite, the vertices of \(\Gamma'\)
are connected to only finitely many vertices in \(X\), say \(v_{1},\dots,v_{n}\).  
Since \(X = \bigcup_{\alpha \in A} X_{\alpha}\), there exist \(\alpha_{1},\dots,\alpha_{n}\)
(not necessarily distinct) so that  \(v_{i}' \in X_{\alpha_{i}}\).  By reordering if necessary, we may assume that
\[
(X_{\alpha_{i}},f_{\alpha_{i}}) \leq
(X_{\alpha_{i+1}},f_{\alpha_{i+1}})
\]
holds for \(i=1,\dots,n-1\).  By definition of the partial order we have that
\[
X_{\alpha_{n}} = \bigcup_{i=1}^{n} X_{\alpha_{i}}
\]
This shows that \(\Gamma'\) is a component of the graph spanned by 
\(V(\Gamma) \setminus X_{\alpha_{n}}\), which is impossible because
\((X_{\alpha_{n}},f_{\alpha_{n}}) \in \mathcal{Z}\).
\end{proof}

Thus we may apply Zorn's lemma and conclude that \(\mathcal{Z}(\Gamma)\)
admits a maximal element.

\medskip
\begin{clm}
\label{third_claim}
If \((X,f) \in \mathcal{Z}\) is a maximal element that \(X = V(\Gamma)\)
\end{clm}

\begin{proof}[Proof of Claim~\ref{third_claim}]
Suppose that \(X \neq V(\Gamma)\) and let \(\Gamma'\) be a connected component of the graph
spanned by \(V(\Gamma) \setminus X\).  Then \(\Gamma'\) is infinite by definition of \(\mathcal{Z}\),
and clearly the degree of any vertex of \(\Gamma'\) is at most its degree as a vertex of  \(\Gamma\) and hence
at most \(d\).  By Claim~4 (applied to \(\mathcal{Z}(\Gamma')\)) we see that there is 
\((X',f') \in \mathcal{Z}(\Gamma')\).  It is clear that \((X \cup X',F)\) is in \(\mathcal{Z}\),
where \(F\) is defined by setting \(F(x) = f(x)\) for \(x \in X\) and
\(F(x) = f'(x)\) for \(x \in X'\).  As \((X,f) \leq (X\cup X',F)\) and \((X,f) \neq (X\cup X',F)\),
we have that \((X,f)\) is not a maximal element of \(\mathcal{Z}(\Gamma)\).
\end{proof}

This completes the proof of Theorem~\ref{translationLikeActionThm}.
\end{proof}

\section{Shortlex shellings}\label{section:shortlexShellings}

Our goal in this section is to define shortlex shellings (Definition \ref{definition:shortlexshelling}) and show that they are parameterized by an SFT (Proposition \ref{proposition:LocalXisSFT}), much in the 
style of Coornaert and Papadopoulos~\cite[\S 3,4]{CoornaertPapadopoulosBook}
or Gromov \cite[\S 7.5, 7.6, 8.4]{Gromov87}.
A shortlex shelling assigns some data to 
each element of $G$. These data impose two simultaneous, compatible structures on $G$: a decomposition into 
horospherical layers (i.e., layers which are locally modeled on spheres in $G$), and a spanning forest locally 
modelled on the tree of shortlex geodesics. 

\pparagraph{Notation.} If $a,b\in \Aa$ and $w\in\genset^*$, we write $a{\buildrel w\over\To} b$ if the shortlex machine, starting in state $a$, ends up in state $b$ after reading $w$. Given $P:G\To G$, and $S\subset G$, let $$P^{-n}S=\{g\in G:P^n(g)\in S\}$$ (as expected) and denote $$P^{-\ast}S:=
\cup_{n=0}^{\infty} P^{-n}S,$$
which we will call the {\em future cone} of $S$ with respect to $P$. 

Given a function $\sigma:G\To A$ for any set $A$, and $g\in G$, let $\sigma\cdot g$ denote the function $G\To A$ given by $(\sigma\cdot g)(h)=\sigma(gh)$. Given $S\subset G$, the 1-interior of $S$ consists of all $g\in G$ such that $B(1,g)\subset \genset$.

\begin{definition}
A preshelling is a triple $X=(h,\state,P)$, where $h:G\To\Zz$ is a 1-Lipschitz function, $\state$ is a function $G\To \Aa$, and $P:G\To G$ satisfying, for all $g\in G$, $d(g,P(g)){\leq 1}$. Given such an $X$, define $\local{X}$ to be the triple $(\deriv h, \state,\dispP)\in\mioi^\genset\times\Aa \times B(1,1_G)$, where $\dispP(g):=g^{-1}P(g)\in B(1,1_G)$.
\end{definition}

\begin{lem}The set, in $\mioi^\genset\times\Aa \times B(1,1_G)$, of $\local{X}$ such that $X$ is a preshelling is a  SFT, which we denote $\sftPre$.
\end{lem}

\begin{proof}
Similar results appear in \cite{CoornaertPapadopoulosBook} for derivatives of horofunctions, and more generally as \cite[Theorem 3.2]{DavidCohen17} for $k$-Lipschitz functions on finitely-presented groups.

For any $\sigma\in(\mioi^\genset)^G$, we may ``integrate'' $\sigma$ along any path $\gamma$ by summing $\sigma(\gamma(n)):\genset\to \mioi$ applied to $\gamma(n)^{-1}\gamma(n+1)$.

 If $\sigma$ integrates to $0$ around any translate of any relator in $G$ then $\sigma$ is the derivative of a $1$-Lipschitz function which can be found by integrating from the identity. 

The group $G$ (being $\delta$-hyperbolic) has a presentation with  generators $\genset$ and relators of length less than or equal to $8\delta+1$, which each fit within $B(4\delta+1,1_G)$. 

There are only finitely many distinct $\local{X}\cdot g|_{B(4\delta+1,1_G)}$, which we take as our allowed cylinder sets defining a subshift $\sftPre$ of finite type. By definition each $\local{X}$ is within $\sftPre$. Moreover if $\phi\in\sftPre$, then the first coordinate of $\phi$ integrates to $0$ around any relator and hence is the derivative of a $1$-Lipschitz function $G\to\Zz$. There are no particular restrictions on the last two coordinates in a preshelling and so $\sftPre$ is the set of all  $\local{X}$ such that $X$ is a preshelling.\end{proof}

\begin{definition}
Let $X_0=(h_0,\state_0,P_0)$, where $h_0:G\To\Zz$, $\state_0:G\To\Aa$ and $P_0:G\To G$ are given as follows.
\begin{itemize}
\item For $g\in G$, $h_0(g)=d(g,1_G)$.
\item If $w\in\genset^*$ is the shortlex minimal word representing $g\in G$, and $a_0$ is the initial state of the shortlex machine, then $\state_0(g)$ is the unique element of $\Aa$ such that $a_0{\buildrel w\over\To}\state_0(g)$ in the notation given at the start of this section.
\item Finally, $P_0(1_G)=1_G$ and for $g\neq 1_G$, $P_0(g)$ is the vertex preceding $g$ in the shortlex geodesic from $1_G$ to $g$. That is, $P_0(g)=h$ if and only if $(\state h)\stackrel{h^{-1}g}{\longrightarrow}(\state g)$.
\end{itemize}
\end{definition}

A shortlex shelling is a preshelling which is locally modelled by $X_0$ in the following sense.
\begin{definition}
\label{definition:shortlexshelling}
A preshelling $X=(h,\state,P)$ is said to be a shortlex shelling if, for every 
$g\in G$ and $R>0$ there exists $g_0\in G$ such that we have the equality of restrictions
$$(\local{X}\cdot g)|_{B(R,1_G)}=(\local{X_0}\cdot g_0)|_{B(R,1_G)},$$
and, furthermore, $B(R,g_0)$ does not contain the identity $1_G$.
\end{definition}

For a preshelling $X$, if $(\local{X}\cdot g)|_F=(\local{X_0}\cdot g_0)|_F$ for some $F\subset G$, we say that $\local{X}$ is modelled by $\local{X_0}$ on $gF$. In other words, $X$ being a shortlex shelling means that $\local{X}$ is modelled by $\local{X_0}$ on every {\it finite} subset of $G$. If $X=(h,\state,P)$ is a shortlex shelling, then $h$ is a horofunction (by definition of horofunction).

We will show that the set of $\local{X}$ such that $X$ is a shortlex shelling is formed by intersecting the preshelling 
SFT with further cylinder sets of radius $2\delta$; hence it is clear that it is a SFT.
We will now show that it is non-empty, and that it includes exactly the shortlex shellings.

\begin{prop}
\label{proposition:LocalXisSFT}
The collection of $\local{X}$ such that $X$ is a shortlex shelling forms a non-empty SFT. In particular, a preshelling $X$ will be a shortlex shelling so long as, for every $g\in G$, there exists $g_0\in G\setminus B(2\delta,1_G)$ such that
$$(\local{X}\cdot g)|_{B(2\delta,1_G)}=(\local{X_0}\cdot g_0)|_{B(2\delta,1_G)}.$$
\end{prop}

\begin{proof}
Let $X=(h,\state,P)$ be a preshelling satisfying the given condition (that $\local{X}$ is modelled by $\local{X_0}$ on $2\delta$-balls not containing $1_G$). We wish to show that $X$ is actually a shortlex shelling, i.e., that on any $B(R,g)$, $\local{X}$ is modelled by $\local{X_0}$. We will proceed by two steps. First, we show that $\local{X}$ is modelled by $\local{X_0}$ on the 1-interior of sets of the form $P^{-\ast}B(2\delta,g)$. Second, we show that every ball $B(R,g)$ is contained in the 1-interior of some cone. Finally we show  the existence of  a shortlex shelling $X$. 

\pparagraph{State determines future.} Given $g\in G$,  since $\local{X}|_{B(2\delta,g)}$ is modeled on a ball away from $1_G$, 
it is clear that
$$\{(g^{-1}g',\state(g')):P(g')=g\}=\{(s,b)\in\genset\times\Aa:\state(g){\buildrel s\over\To}b\}.$$
Now, suppose that $\state(g)=\state_0(g_0)$ for some $g,g_0\in G$. We observe by induction that $g'\in P^{-\ast}\{g\}$ if and only if, for the shortlex geodesic representative  $w\in\genset^*$  of $g^{-1}g'$, 
 $\state(g){\buildrel w\over\To}\state(g')$.
  It follows that
$$g^{-1}P^{-*}\{g\}=g_0^{-1}P_0^{-*}\{g_0\}.$$
Furthermore, for $g'\in P^{-*}(g)$, we have $\state(g')=\state_0(g_0g^{-1}g')$ and $\dispP(g')=\dispP_0(g_0g^{-1}g')$, because the $\state(g)$ and $g^{-1}g'$ uniquely determine $w$ as above. Equivalently, we have shown that
$$((\state,\dispP)\cdot g)|_{g^{-1}P^{-*}\{g\}}=
((\state_0,\dispP_0)\cdot g_0)|_{g_0^{-1}P_0^{-*}\{g_0\}}.$$
Finally, for $g'\in P^{-*}g$, with $w$ as above, we have
$$h(g')-h(g)=\ell(w)=h_0(g_0g^{-1}g')-h_0(g_0),$$
or, equivalently,
$$(h\cdot g)|_{g^{-1}P^{-*}\{g\}}=
(h_0\cdot g_0)|_{g_0^{-1}P_0^{-*}\{g_0\}}+h(g)-h(g_0).$$

\pparagraph{On the 1-interior of cones, $\local{X}$ is modelled by $\local{X_0}$.} Let $g,g_0\in G$ and suppose that
$$(\local{X}\cdot g)|_{B(2\delta,1_G)}=(\local{X_0}\cdot g_0)|_{B(2\delta,1_G)}.$$
By the above considerations, we have that
$$g^{-1}P^{-*}B(2\delta,g)=g_0^{-1}P_0^{-*}B(2\delta,g_0),$$
and, furthermore,
$$((\state,\dispP)\cdot g)|_{g^{-1}P^{-*}B(2\delta,g)}=
((\state_0,\dispP_0)\cdot g_0)|_{g_0^{-1}P_0^{-*}B(2\delta,g_0)}$$
and
$$h|_{g^{-1}P^{-*}B(2\delta,g)}=
h_0|_{g_0^{-1}P_0^{-*}B(2\delta,g_0)}+h(g)-h_0(g_0).$$
Consequently, $\local{X}$ is modelled by $\local{X_0}$ on the 1-interior of $P^{-*}B(2\delta,g)$.

\pparagraph{Every ball lies in the 1-interior of some cone.}
For every $R>0$,  $g\in G$, $n\geq R+\delta+1$, and $g'\in P_0^{-n}(g)$, we claim that
$$P_0^{-*}B(\delta,g)\supset B(R,g').$$

To see this, for any $x\in B(R,g')$, consider the geodesics along $\{P_0^i(g')\}$ and $\{P_0^i(x)\}$ from 
$g'$ and $x$ to $1_G$. Since $g$ is in $\{P_0^i(g')\}$ and
$\dist(g,g')=n\geq R+\delta+1$ and $\dist(x,g')\leq \delta$, by the triangle inequality, every point on any 
geodesic between $x$ and $g'$ must be of distance greater than $\delta$ from $g$. By the $\delta$-slim 
triangle condition,  some point on the geodesic from $x$ to $1_G$ is within $\delta$ of $g$, and so $x$ is in 
$P_0^{-*}B(\delta,g)$.

It follows that for all $g\in G$, and  $n\geq R+\delta+1$,
$$B(R,g)\subset P^{-*}B(\delta,P^n g).$$

\pparagraph{$X$ is a shortlex shelling.} If $n\geq R+\delta+2$, we see from the above that 
$\local{X}|_{B(R,g)}$ is modelled by $\local{X_0}$. It follows that $X$ is a shortlex shelling.

\pparagraph{There exists a shortlex shelling.}
Let $\{g_n\}$ be a sequence in $G$ with $\dist(g_n,1_G)=n$. By compactness 
$\{(\local{X_0}\cdot g_n) |_{B(n,1_G)}\}$ has a  subsequence that converges to a shortlex shelling. 
\end{proof}

\begin{cor}
\label{corollary:predecessorpaths}
From the proof, we see that for any $g,g'\in G$, the geodesics $\gamma:n\mapsto P^n(g)$ and $\gamma':n\mapsto P^n(g')$ satisfy $\liminf\dist(\gamma(n),\gap(n))\leq 2\delta$, which implies that they are asymptotic by Lemma \ref{lemma:thinquads}.
\end{cor}
 
We now give a name to the SFT formed by local data of shortlex shellings.

\begin{definition}
Let $\sftS$ denote the set of all $\local(X)$ such that $X$ is a shortlex shelling.
\end{definition}

We note that this SFT always has configurations with infinite order periods. The rest of the paper revolves around ``populated shellings'', which are shortlex shellings decorated with some extra data that kills these periods.

\section{The measure $\mu$}\label{Section:InvariantMeasure}

In this section we prove Proposition \ref{proposition:dense}, which shows that there is a function $\mu:\Aa\To[0,\infty)$ such that for any shortlex shelling $X$, $\mu\circ\state$ is positive on a dense (in the sense of Definition \ref{definition:dense}) set of points, and the sum of $\mu\circ\state$ over the successors of $g\in G$ is equal to $\lambda\mu(\state(g))$. This regularizes the growth of $P^{-1}$---in particular, for a finite $S\subset G$, we see that although $P^{-1}(S)$ may not have cardinality equal to $\lambda\# S$, we still have that $\mu$ assigns exactly $\lambda$ times as much mass to $P^{-1}(S)$ as it does to $S$. This, in turn will be crucial in showing that populated shellings defined in \S\ref{Section:PopulatedShellings} exist and have no infinite order periods.

Recall that Proposition~\ref{proposition:eigenvector} gives a left eigenvector of $[\Mm]$ with eigenvalue 
$\lambda$ (where \(\lambda\) is the growth rate of \(G\) with the generators \(\genset\)), 
supported on states of \em maximal growth, \em that is, the states denoted by
\(\Aabig \cup \Aamax\) in Proposition~\ref{proposition:eigenvector}.

\begin{definition}
Let $\mu:\Aa\To[0,\infty)$ be the function given $\mu(a_i)=\mu_i$, where $\mu$ is the left eigenvector defined in Proposition \ref{proposition:eigenvector}, normalized so that the smallest nonzero value of $\mu$ is $1$. Given a fixed shortlex shelling $X=(h,\state,P)$ and $g\in G$, $\mu(g)$ is understood to be $\mu(\state(g))$
 \end{definition}

\begin{rmk}\label{remark:sum} Consequently, from the definitions of shortlex shelling and $\mu$: 
$$\sum_{b:P(b)=a} \mu(b) =\lambda\; \mu(a)$$ This is the key property of $\mu$ which will be exploited in the proof of Proposition~\ref{prop:Balancing}, the existence of ``populated shellings''.  \end{rmk}

\begin{definition}
Let $G^+$ to consist of all $g\in G$ with $\mu(g)>0$.  For any horosphere $H$, let $H^+ := H\cap G^+$.
 \end{definition}

\begin{definition}[\(k\)-dense]\label{definition:dense}
Let \(G\) be a metric space and $G'\subseteq G$. We say that $G'$ is $k$-dense in $G$ if for all $g\in G$ there exists $g'\in G'$ such that
\(\dist(g,g') \leq k\). 
\end{definition}

\begin{prop}
\label{proposition:dense}
For any shortlex shelling $X$, the set $G^+$ is $2\delta$-dense.
\end{prop}
\begin{proof}
In the proof of this proposition, in order
to be consistent with the \em left \em action of \(G\) on \(\partial G\), we will consider the
\em left \em action of \(G\) on \(\sftS\) given by
\[
(g \cdot \omega)(g') = \omega (g^{-1} g')
\]

We will proceed as follows. First we describe a factor map $\pi:\sftS\To\partial G$. We will use this map, together with the fact the $\partial G$ is minimal as a $G$-system, to show that every shortlex shelling includes states from $\Aamax\cup\Aabig$. We then use a compactness argument to show that there exists a $k$ such that $\Aamax\cup\Aabig$ states are $k$-dense in every shelling. Finally we will use the fact that the future of any $2\delta$-ball contains a $k$-ball to conclude that such states are $2\delta$-dense.

\pparagraph{Coding the boundary.} Given a shortlex shelling \(X=(h,\state,P)\), consider $\local X\in\sftS$.
The function $\gamma_X:n\mapsto P^n(1_G)$ satisfies \(h \circ \gamma_{X}(n) = -n\)
and therefore defines a geodesic ray.  This defines a map \(\pi:\sftS \to \partial G\) 
\[
\pi:\local X \mapsto [\gamma_{X}]
\]
We claim that \(\pi\) is a factor map, that is, \(\pi\) is continuous, equivariant, 
and surjective.

Continuity follows directly from the definitions.  To see that \(\pi\) is equivariant, 
fix \(g \in G\) and let \(g \cdot X := (h',\state',P')\), so that $\dispP'(g')=\dispP(g^{-1}g')$. We have that $\gamma_{g\cdot X}(n)=gP^n(g^{-1})$ because a simple induction shows that
$$P'^{n}(1_G)=gP^{n}(g^{-1}),$$
since $P'^0(1_G)=gP^0(g^{-1})$ and the inductive hypothesis $P'^{n}(1_G)=gP^{n}(g^{-1})$ implies $$ 
P'^{n+1}(1_G)=P'^{n}(1_G)\dispP'(P'^{n}(1_G))=gP^{n}(g^{-1})\dispP'(P'^{n}(1_G))$$ $$
 =gP^{n}(g^{-1})\dispP(P^{n}(g^{-1}))
 =gP^{n+1}(g^{-1}).$$
By Corollary \ref{corollary:predecessorpaths}, we know that $\gamma:n\mapsto P^n(g^{-1})$ is asymptotic to $\gamma_X$, and thus
$$\pi(g\cdot X)=[\gamma_{g\cdot X}]=[g\cdot \gamma]=g\cdot\pi(X),$$ showing that $\pi$ is $G$-equivariant. Finally, by \cite{Gromov87}, the action of $G$ on its boundary is minimal, so the image of $\pi$ must be all of $\partial G$, since it is a closed, nonempty subset of $\partial G$ preserved by $G$.

\pparagraph{Every shortlex shelling includes a state of maximal growth.} Let $\sftS'$ consist of all $\local X$ such that $X=(h,\state,P)$ is a shortlex shelling with $\state(G)\subseteq\Aamin.$ We wish to show that $\sftS'$ is empty, so suppose otherwise. By minimality of $\partial G$ and the fact that $\pi$ is a factor map, we see that every point of $\partial G$ may be represented by an element $\pi(\sftS')$.

Since $\state_0$ realizes values in $\Aabig$ at infinitely many points, by compactness, there exists a shortlex shelling $X=(h,\state,P)$ such that $\state(1_G)\in\Aamax\cup\Aabig$. Let $X'=(h',\state',P')$ be a shortlex shelling such that $\local X'\in \sftS'$ and $\pi(\local X)=\pi(\local X')$. For $g\in P^{-n}(1_G)$, we may form asymptotic geodesics $\gamma,\gamma'$ based at $g$ via $\gamma(n)=P^n(g)$ and $\gamma'(n)=P'^n(g)$ and apply Lemma \ref{lemma:asymptotic} to see that $P'^n(g)$ is within $2\delta$ of $1_G$. Hence, $\#P'^{-n}B(2\delta,1_G)\geq \#P^{-n}\{1_G\}.$ Since $\state(1_G)\in\Aamax\cup\Aabig$, we know
$$\log(\#P^{-n}\{1_G\})/n\To \log(\lambda),$$
but since $\state'(B(2\delta,1_G))\subset\Aamin$, we have
$$\limsup\log(\#P'^{-n}B(2\delta,1_G))/n<\log(\lambda),$$
giving us a contradiction. We conclude that $\sftS'$ must be empty.

\pparagraph{Maximal growth states are $k$-dense for some $k$.} Finally, suppose there is no $k$ such that states of $\Aamax\cup\Aabig$ occur $k$-densely in every $\local X\in \Omega$. Then there exist shortlex shellings $X_k=(h_k,\state_k,P_k)$ and $g_k\in G$ such that $\state_k(B(k,g_k))\subset \Aamin$. Then $g_k^{-1}\cdot\local X_k$ subconverges to some $\local X\in \sftS$, but we must have $\local X\in\sftS'$, which we have seen is impossible.

\pparagraph{Maximal growth states are $2\delta$-dense.} Suppose that $\state(B(2\delta,g))\subseteq\Aamin$. We have seen in the proof of Proposition \ref{proposition:LocalXisSFT} that there exists some $g'\in P^{-*}(g)$ such that $B(k,g')\subseteq P^{-*}B(2\delta,g)$. Since $\Aamin$ states, by definition, can only lead to $\Aamin$ states, we have $\state(B(k,g'))\subseteq\Aamin$. Because $G^+$ is $k$-dense, we know that this cannot be the case, so we conclude that $G^+$ is in fact $2\delta$-dense.
\end{proof}

\pparagraph{Finding dense states.} We remark that, for any subshift $\Omega\subset A^G$ on a finitely generated group, there exists $B\subset A$ and $k\in \Nn$ such that $\Omega\cap B^G$ is nonempty and for all $b\in B$ and $\omega\in \Omega\cap B^G$, the $\omega^{-1}(b)$. To see this, simply take a minimal $B$ such that $\Omega\cap B^G$ is nonempty---if, for all $k$, there were an $\omega_k\in\Omega\cap B^{G}$ such that the symbol $b\in B$ did not occur in $\omega_k|_{B(k,g_k)}$, then $\omega_k\cdot g_k$ would subconverge to a configuration in $(B\setminus b)^G \cap \Omega$. Furthermore, if $\Omega$ is an SFT, so is $\Omega\cap B^G$. This construction may be used instead to enforce density.

\section{The Divergence Graph on Horospheres}\label{section:divergenceMetric}
For horospheres $H$ of a shortlex shelling $X=(h,\state,P)$, we now construct a graph with vertices $H^+:=H\cap G^+$ which behaves nicely with respect to $P$ in the sense that predecessors of neighboring vertices either coincide or are neighbors; and each pair of adjacent vertices admits a pair of adjacent successors. In other words each edge has a predecessor (in the previous horosphere) that is an edge or a vertex, and each edge has at least one successor edge (in the next horosphere). We call this graph the divergence graph on $H$ and show that its vertex set is dense in $H$ (Lemma \ref{Lemma:H+is2deltadense}), that its edges have bounded length in the word metric (Lemma \ref{LemmaEdgesOfLength2Delta}) and that it is connected (Lemma \ref{LemmaDivGraphConnected}). When we define populated shellings in the next section, we will require that a child of a person living at $v\in H^+$ must live in a village $u\in G$ whose predecessor $P(u)\in H^+$ lies close to $v$ in the divergence graph on $H$. All of the facts noted here will be needed.

 \begin{lem}\label{Lemma:H+is2deltadense} For any horosphere $H$ in a shortlex shelling $X$, $H^+$ is $4\delta$-dense in $H$.
 \end{lem}
 
 \begin{proof}
 Let $v$ be a element of $H$ and let $B$ be the $2\delta$ ball in $G$ around $v$. The future of $B$ contains arbitrarily large balls, and in particular must contain elements of $G^+$. Thus $B$ contains an element of $G^+$, say $v'$. Now $v'$ must have either a predecessor or successor $v''$ in $H^+$. We have that
 $$\dist(v,v'')=|h(v)-h(v')|\leq 2\delta$$
 and thus $\dist(v'',v)\leq 4\delta.$ 
 \end{proof}
 
 \begin{definition}\label{DefinitionDisth}
The  {\em divergence graph} on $H$ has vertices $H^+$ and has an edge between $g_1$ and $g_2$ if and only if 
there exists $C$ such that for all $n\in\Nn$, $\dist(P^{-n}\{g_1\},P^{-n}\{g_2\})<C$. 
\end{definition}

In Lemma~\ref{LemmaDivGraphConnected} below, we  show the divergence graph is connected. 
The following lemma shows that if the futures of two points in a horosphere remain bounded distance apart, then the points and their futures are within $2\delta$ of one another and that valence in a divergence graph is bounded.

\begin{lem}\label{LemmaEdgesOfLength2Delta} Let $g_1,g_2$ be in some  $H^+$. If there exists $C>0$ such that for all $n\geq 0$, $\dist(P^{-n}\{g_1\},P^{-n}\{g_2\})<C$ then for all $n\geq 0$, $\dist(P^{-n}\{g_1\},P^{-n}\{g_2\})\leq 2\delta$. In particular if $g_1$ and $g_2$ are connected by an edge in a divergence graph then $\dist(g_1,g_2)\leq 2\delta$, and so the valence of a vertex in a divergence graph is bounded by the size of $B(2\delta,1_G)$.
\end{lem}

\begin{proof}
Suppose for $g_1$ and $g_2$  in $H^+$, there is some $C$ with $\dist(P^{-n}\{g_1\},P^{-n}\{g_2\})<C$ for all $n$. Take some $n>C+2\delta$. There exists some $g_0\in G$ such that $h(g')=\dist(g',g_0)-C$ for all $g'$ in a containing $g_1,g_2,P^{-n}{g_1}, P^{-n}{g_2}$. Let $\gamma_i$, $i=1,2$ be a geodesic from $g_0$ to $g_i$.
Let $t=\dist(g_0,g_1)=\dist(g_0,g_2)$.
Then by  Lemma~\ref{lemma:thintriangle}, $\dist(g_1,g_2)=\dist(\gamma_1(t),\gamma_2(t))\leq 2\delta$.\end{proof}

\begin{lem} \label{LemmaDivGraphConnected}
If $H$ is a horosphere in a shortlex shelling admitted by $\Mm$, then the divergence graph on $H^+$ is connected.
\end{lem}

\begin{proof}
Let $X=(h,\state,P)$ be a shortlex shelling admitted by $\Mm$. Without loss of generality, set $H=h^{-1}(0)$, and let $\xi$ denote the point of $\partial G$ represented by the geodesic ray $n\mapsto P^n(1_G)$. A deep result of Swarup (building on work of Bowditch) asserts that $\partial G\setminus \xi$ is connected because $G$ is one-ended~\cite{swarup}. We will use this to show that the divergence graph on $H^+$ is connected. The following definitions relate these two spaces.

\begin{itemize}
\item By an $X$-geodesic, we mean any geodesic ray $\gamma$ in $G$ such that for all $n$, $h\circ\gamma(n)=n$, $\gamma(n)=P(\gamma(n+1))$, and $\gamma(n)\in G^+$.
\item If $S$ is a subset of $H^+$, let $\proj(S)$ denote the subset of $\partial G$ consisting of all $[\gamma]$ where $\gamma$ is an $X$-geodesic with $\gamma(0)\in S$. (We write $\proj(v)$ for $\proj(\{v\})$.)
\end{itemize}

Let $S$  be any  component  of the divergence graph in $H^+$. We are going to show that $\proj(S)$ and $\proj(H^+\setminus S)$ disconnect $\partial{G}\setminus\xi$ unless $H^+\setminus S$ is empty. We claim the following conditions are satisfied: 

(1) $\proj(g)\neq \emptyset$ for any $g\in H^+$. Let $g_n$ be a point in $P^{-n}(g)\cap G^+$ (which is nonempty by definition of $\mu$.) Let $\gamma_n$ be the geodesic path given by $t\mapsto P^{n-t}(g_n)$. Then $\gamma_n(0)=g$ for all $n$, and the $\gamma_n$ subconverge by Lemma \ref{lemma:arzela}. Clearly this limit is an $X$-geodesic.

(2) {$\proj(H^+)=\partial{G}\setminus\xi$:}
Comparing $h$ values tells us that no $X$-geodesic is asymptotic to $n\mapsto P^n(1_G)$, so $\xi\notin\proj(H^+)$. Let $\eta\in \partial{G}\setminus\xi$. We must show that $\eta$ is represented by an $X$-geodesic. Some biinfinite geodesic $\gamma$ connects $\eta$ and $\xi$ \cite[Lemma III.H.3.2]{BridsonHaefliger99}, which we parametrize so $h(\gamma(0))=0$.  By Lemma~\ref{Lemma:H+is2deltadense}, for any $n\in\Zz$, there is some $g_n\in G^+$ so that $\dist(g_n, \gamma(n))\leq 2\delta$. Let $\gamma_n$ be an $X$-geodesic such that $\gamma_n(h(g_n))=g_n$ (we see that these exist by (0)). By Lemma \ref{lemma:asymptotic} applied to the reverse of $\gamma$ and $\gamma_n$, $\dist(\gamma(0),\gamma_n(0))$ is bounded, so the $\gamma_n$ subconverge to some $X$-geodesic $\gamma'$ by Lemma \ref{lemma:arzela}. By \ref{lemma:independence}, $\gamma'$ is asymptotic to $\gamma$, and thus $\eta\in\proj(H^+)$.

(3)  $\proj(S)\cap\proj(H^+\setminus S) = \emptyset$:
By our definitions, any points $p, q\in H^+$ with $\proj(p)\cap\proj(q)\neq \emptyset$ share an edge in the divergence graph, and $S$ is a component.

(4) 
Let $(\gamma_n)$ be a sequence of $X$-geodesics such that   the sequence
 $(\eta_n:=[\gamma_n])$  is in $\proj(H^+)=\partial G\setminus\xi$ and converges to some $\eta=[\gamma]\in\proj(H^+)$, for some $X$-geodesic $\gamma$. Recall $\{\gamma_n(0)\} \subset H^+$; we
claim that this set is finite. 

 Assume, for a contradiction, that it is infinite and
fix some $x > 4\delta$. Since $\{\gamma_n(0)\}$ is infinite, then the set $J_x:=\{n\in\Nn: \dist(\gamma_n(0), \gamma(0)) > 2x+4\delta\}$ is infinite. 
For each $n\in\Nn$, by \cite[Lemma III.H.3.1]{BridsonHaefliger99}
 let $\gap_n$ be any geodesic with  $\gap_n(0)=\gamma(0)$ and $[\gap_n]=[\gamma_n]=\eta_n$. By Lemma~\ref{lemma:arzela} by subsequencing we may assume $\gap_n$ converges to $\gap$. By Lemma~\ref{lemma:independence} $\gap$ and $\gamma$ are asymptotic. We will obtain a contradiction because our choice of $x$ forces $\gap_n$ to dip below the horosphere $H$ while $\gap$ must fellow-travel with the $X$-geodesic $\gamma$. 
 
 \psfrag{A}[r][r]{$\gamma_n(p)$}
\psfrag{B}[r][r]{$\not{\!y}$}
\psfrag{C}[r][r]{$\gamma_n(0)$}
\psfrag{D}[l][l]{$\gamma(0)=\gap(0)$}
\psfrag{E}[l][l]{$\gap_n(x)=\gap(x)$}
\psfrag{F}[l][l]{$\gap_n(q)$}
\psfrag{G}[c][c]{$\tilga$}
\psfrag{H}[c][c]{}
\psfrag{I}[c][c]{}
\psfrag{i}[r][r]{}
\psfrag{J}[c][c]{$\hat\gamma$}
\psfrag{X}[l][l]{}
\psfrag{y}[c][c]{$y$}
\centerline{\includegraphics[scale=.7]{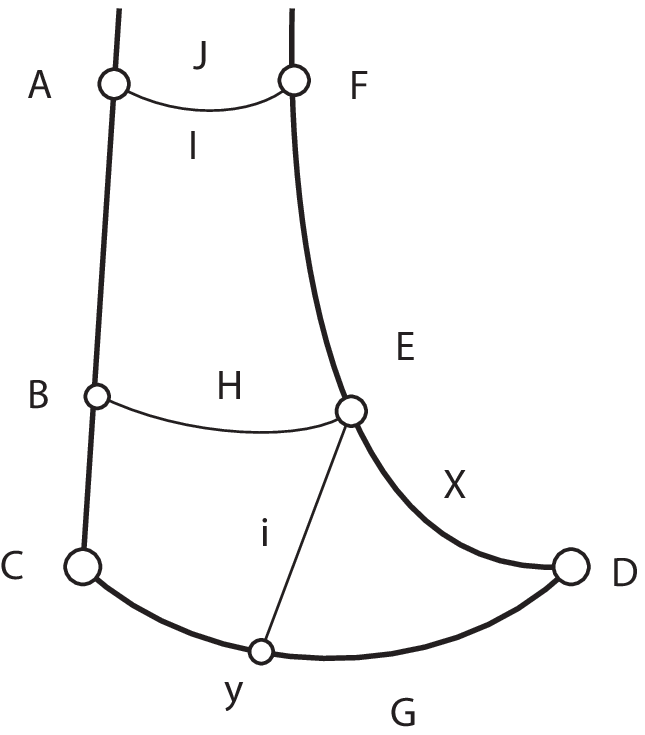}}

There exists some $n$ such that $\gap_n(x)=\gap(x)$.   Let \(\tilga\) be a geodesic connecting
\(\gamma(0)\) to \(\gamma_{n}(0)\).  By Lemma~\ref{lemma:asymptotic} for sufficiently large 
\(q\), there exists \(p\), so that \(\dist(\gamma_{n}(p),\gamma_{n}'(q)) \leq 2\delta\).
Let \(\hat\gamma\) a geodesic connecting
\(\gamma_{n}(p)\) to \(\gamma_{n}'(q)\); we take \(q > 4\delta + x\).   
By the slim quads condition, \(\gamma_{n}'(x)\) is within \(2\delta\) from some point \(y\)
on one of the other three sides,
and since \(q > 4\delta+x\) we have that \(y \not \in\hat\gamma\).
We claim that \(y \not\in \gamma_{n}\); assume, for a contradiction, that it is.
Since \(\dist(\gamma(0),y) \leq x + 2\delta\), we have that \(h(y) \leq x+2\delta\).
On the other hand, since \(\gamma_{n}\) is an \(X\)-geodesic, we have that
\[
h(y) = \dist(\gamma_{n}(0),y) \geq
\dist(\gamma_{n}(0),\gamma(0)) - \dist(\gamma(0),y) >
2x + 4\delta - (x + 2\delta) =
x + 2\delta
\] 
This contradiction shows that \(y \not \in \gamma_{n}\).

Therefore \(y \in \tilga\).  By Lemma~\ref{lemma:ladder} we have that
\(\dist(\gamma_{n}'(x),\tilga(x)) \leq 4\delta\).
By Lemma~\ref{lemma:DipSomeMore} we have that 
\(h(\tilga(x)) \leq 2\delta -x\); therefore 
\(h(\gamma_{n}'(x)) \leq 6\delta - x< 2\delta\).
On the other hand, \(\gamma_{n}'(x) = \gap(x)\).  Since
\(\dist(\gap(x),\gamma(x)) \leq 2\delta\), we have that
\[
h(\gap(x)) \geq h(\gamma(x)) - 2\delta = x - 2\delta > 2 \delta
\]
a contradiction, showing that $\{\gamma_n(0)\}$ is finite.

(5) For  $A=S$ or $A=H^+\setminus S$,  $\proj(A)$ is closed in $\partial G$:  
Given a sequence $(\eta_n)\subset\proj(A)$ converging to some $\eta\in\proj(H^+)$, we wish to show that $\eta\in \proj(A)$. Represent each $\eta_n$ with an $X$-geodesic $\gamma_n$ with $\gamma_n(0)\in A$. By (4), $\{\gamma_n(0)\}$ is finite and therefore $\gamma_n$ subconverges to some $\tilde\gamma$ with $\tilde\gamma(0)\in \{\gamma_n(0)\}\subset A$. By Lemma~\ref{lemma:independence}, $[\gamma]=[\tilde\gamma]\in \proj(A)$. 
In other words,  $S$ and $H^+\setminus S$ are closed.

As noted above, ~\cite{swarup} shows that $\partial{G}\setminus\xi$ is connected. Consequently, by  (2), (3) and (5), one of $\proj(H^+\setminus S)$ or $\proj(S)$ is empty. By (1) $\proj(S)$ is not empty and  so  $H^+\setminus S=\emptyset$. In other words, the divergence graph on $H^+$ is connected.
\end{proof}

\section{Populated Shellings}
\label{Section:PopulatedShellings}

In the remainder of the construction we consider the divergence graphs on $H^+$ for each horosphere $H$ in each shortlex shelling on $G$.
By Lemma~\ref{LemmaDivGraphConnected} the divergence graph is connected (since  $G$
is one-ended), and 
by Lemma~\ref{LemmaEdgesOfLength2Delta} the degree of the divergence graph is at
most  $B(2\delta,1_G)$.   Hence by Theorem~\ref{translationLikeActionThm}  the divergence
graph admits a translation-like $\mathbb{Z}$ action, say given by $\psi:H^+ \to H^+$, 
with defect $L$ where $L: =2\#B(2\delta,1_G)+1$ does not depend on choice of $H$ or $X$. We fix this $L$ for the remainder of the paper and note that $L>2\delta$.  This translation-like $\mathbb Z$ action $\psi$ will be central to our proof in Section~\ref{subsection:existence} of the existence of ``populated shellings'', defined below.

For the following, we define, for any $K\in\Nn$ and $R\subset H^+$, the set $\nbhd_K(R)$, $R\subset \nbhd_K(R)\subset H^+$ of points connected to $R$ by paths in the divergence graph on $H^+$ of length no greater than $K$.
We abbreviate $\nbhd_L(R)$ as $\nbhd R$. 
Since by Lemma~\ref{LemmaEdgesOfLength2Delta}, points connected by edges in a divergence graph are at most $2\delta$ apart, we observe:

\begin{lem}\label{LemmaBoundsOnNL} On any $H^+$ in any shortlex shelling, for any $K\in\Nn$, 
for any $R\subset H^+$,  $\nbhd_K(R)$ is contained within a $2\delta K$ neighborhood of $R$ 
{(in the word metric)}. In particular,  $\nbhd R$ is contained within a 
$2\delta(2\#B(2\delta,1_G)+1) $ neighborhood of $R$. 
\end{lem}

\subsection{Populated shellings}\label{subsection:populatedShellings}

Fix $q\in\{2,3\}$ such that $\log(q)\notin\Qq\log(\lambda)$. 

\begin{definition}
\label{definition:populatedshelling}
\rm
A \em populated shelling \em of $G$ 
(with \em population bound \em $N\in\Nn$ and \em growth \em by powers of $q$) 
is a shortlex shelling equipped with the following extra data:
\begin{itemize}
\item a ``population'' function $\popfunc:G\To [0..N]$; 
\item a ``population density'' function $\density:G\To\{\lfloor\log_q(\lambda)\rfloor,\lceil\log_q(\lambda)\rceil\}$ (note $q^\density$ is always in $\Nn$);
\item and a  ``parent-child matching'' function

$$\match:\{(v,j,k)\ \mid \ v\in G, 1\leq j\leq\popfunc(v), 1\leq k\leq q^{\density(v)}\}\to \{(v,j)\ \mid\ v\in G, 1\leq j\leq \popfunc(v)\}$$
\end{itemize}

such that 
\begin{itemize}
\item denoting the coordinates of  $\match=(\match_G,\matchpop)$, for any triple  $(v,j,k)$ in the domain   $\match_G(v,j,k)\in P^{-1}\nbhd (v)$; 
\item $\popfunc(g)=0\Leftrightarrow \mu(g)=0$ (and so the domain of $m$ restricted to  $v$ in $G\setminus G^+$ is empty);
\item $\density$ is constant on horospheres;
\item  and $\match$ is
 a bijection.
\end{itemize}
\end{definition}

For $H$ a level set of $h$, we refer to $\{(v,j):v\in H,j\in [1..\popfunc(v)]\}$ as the set of ``people'' in $H$. We say that person $(v,j)$ ``lives'' at a ``village'' $v$. Each $(v,j)$ has $q^{\density(v)}$ ``children''. For each 
 $k\in [1..  q^{\density(v)}]$, if
$\match(v,j,k)=(w,l)$, then we say that $(w,l)$ is the $k$th ``child'' of $(v,j)$  and conversely $(v,j)$ is the ``parent'' of $(w,l)$. Note that each person has exactly one parent and $ q^{\density(v)}$ children.

\begin{definition}
The local data associated to the populated shelling  $X=(h,\state,P,\popfunc,\density,\match)$ is the function
$$\begin{array}{rc} \local{X}: G\ \  \To &  \mioi^\genset \times \Aa \times \genset
\times [0..N]  \times \{\lfloor\log_q(\lambda)\rfloor,\lceil\log_q(\lambda)\rceil\}   \times M \end{array}$$
given by
$$\local{X}:g\mapsto (\deriv h,\state,\dispP,\popfunc,\density,\dispG
)(g)$$

where $M$ is the finite set of functions  with (possibly empty) domain within $[1..\popfunc(g)]\times [1..q^{\delta(g)}]$ and range $B(2\delta L+1, 1_G)\times [1..N]$. We define $\dispG(g)(j,k)=(g^{-1}\match_G(g,j,k),\matchpop(g,j,k))$.  
\end{definition}

\begin{rmk}  The first coordinate of $\dispG(g)$ lies within $B(2\delta L+1, 1_G)$ by Lemma~\ref{LemmaBoundsOnNL}. 
This bound will be used throughout the remainder of the construction. Moreover  $\dispG$ is the empty function for $g\notin G^+$. The first coordinate of $\dispG(g)$ gives the relative position of the village in which the $k$th child of the 
$j$th villager of $g$ lives, and the second coordinate gives which villager that child is. 
\end{rmk}

\begin{prop}\label{lemma:SFTP}
The set of all $\local{X}$ such that $X$ is a populated shelling 
 forms an SFT, $\sftP$. 
\end{prop}

We will show that this SFT $\sftP$ is non-empty (for sufficiently large $N$) in  Proposition~\ref{Cor:PopShellingsExist} and that 
 the stabilizer of any $\local X\in \sftP$  contains no infinite order element (for appropriately chosen $q$) 
 in  Proposition~\ref{Prop:Aperiodic}.

\begin{proof} Recall Proposition~\ref{proposition:LocalXisSFT} that  {$\sftS$} the set of $\local X$ such that $X$ is a shortlex shelling is an SFT with cylinder sets of size { $4\delta+1$}. 
We will show that the set of $\local X$ such that $X$ is a populated shelling is an SFT by taking cylinder sets of radius $2\delta L+1>4\delta+1$, and show that these local rules are sufficient to enforce the conditions defining the functions $\popfunc, \match$ and $\density$ on a populated shelling. 

Recall that by Lemma~\ref{Lemma:H+is2deltadense} the vertices  of a divergence graph are $4\delta$-dense in its horosphere and by Lemma~\ref{LemmaEdgesOfLength2Delta} the distance between endpoints of a divergence edge is at most $2\delta$. Because the group is one-ended, by 
 Lemma~\ref{LemmaDivGraphConnected},  the divergence graph on each horosphere is connected. Consequently, in order 
to ensure that $\density$ is constant on horospheres it suffices to consider cylinder sets of size at least $4\delta$.  
The conditions on $\popfunc$ and $\match$ are defined within $(2\delta L+1)$-balls and so are ensured by cylinder sets of this size. 
\end{proof}

\subsection{The existence of populated shellings.}
\label{subsection:existence}

Recall, as discussed at the beginning of Section~\ref{Section:PopulatedShellings}, that by Theorem~\ref{translationLikeActionThm}  each divergence
graph admits a translation-like $\mathbb{Z}$ action, $\psi:H^+ \to H^+$, 
with defect  $L =2\#B(2\delta,1_G)+1$.

For each $i\in\Zz$, let $H=H_i$ be the level set $h^{-1}\{i\}$. For convenience, when clear from context we will drop the subscript $i$.

For the remainder of this section, we  fix some  shortlex shelling $(h,\state,P)$, and some  
 For any  $R\subset H^+:=H\cap G^+$, let $\partial R$ denote $\nbhd R \setminus R$, recalling that $\nbhd R:=\nbhd_L R$ is the $L$-neighborhood of $R$ in $H^+$ with distance measured in the divergence graph, $L$ as defined at the beginning of Section~\ref{Section:PopulatedShellings}.

Recall our conventions for summation: we write $f_R := \sum_{x\in R} f(x)$ for 
 sums of  values of some function $f$ over 
set some set $R$;  We may also write $f_{m..n}=\sum_{k=m}^n f(k)$.

\begin{definition} Given $\nu,C > 0$, we say that $\popfunc:H\to\{0,1,\ldots\}$ {\em realizes density $\nu$ up to error $C$} if the following conditions hold.
\begin{itemize}
\item $\popfunc(v)=0\Leftrightarrow \mu(v)=0$.
\item For any finite region $R\subset H$, 
\begin{sumsubscripts} $ \left|\popfunc_{R} - \nu\ 
{
 \mu_{R}}\right|\leq C  \mu_{\vspace{.1em}\partial R}$ \end{sumsubscripts}. 
\end{itemize}
\end{definition}

\begin{lem}\label{densityOnHorofunctions}
For any  $\nu>1$ there exists a function
$$\popfunc:H\to\{0\}
\cup[\lfloor\nu\rfloor..\lceil\nu\max_{a\in\Aa}\mu(a)\rceil]$$ that realizes $\nu$ with error 2.

\end{lem}

\begin{proof} By Proposition~\ref{translationLikeActionThm}, there is a $\Zz$-action $\psi$ on the divergence graph in $H^+$, with $\dist(\psi(g),g)\leq L$. 
Let $\Lambda\subset H^+$ be a set of orbit representatives and for each 
 $\alpha\in\Lambda$  define $p_\alpha:\Zz\to H^+$ as $p_\alpha(n)=\psi^{n}(\alpha)$. Of course the  images of these $p_\alpha$ are disjoint and cover $H^+$.
 
 \def\oo{{\ast}}
Choosing arbitrary basepoint $\oo\in\Rr$, we define $\popfunc$ on $H$: 
On $H\setminus H^+$, 
 we define $\popfunc$ to be identically $0$.  
 Each $v\in H^+$ may be written uniquely as  some   $p_\alpha(n)$. As in the illustration below, we define $\popfunc$ on  $H^+$, abbreviating $\popfunc(p_\alpha(n))$, $\mu(p_\alpha(n))$ as $\popfunc_n$ and $\mu_n$. 
 
 \begin{sumsubscripts}
\psfrag{a}[r][r]{$\oo\in\Rr$}
\psfrag{b}[r][r]{$+\nu\,\mu_2$}
\psfrag{c}[r][r]{$+\nu\,\mu_1$}
\psfrag{d}[r][r]{$+\nu\,\mu_0$}
\psfrag{e}[r][r]{$-\nu\,\mu_{-1}$}
\psfrag{f}[r][r]{$-\nu\,\mu_{-2}$}
\psfrag{g}[r][r]{$-\nu\,\mu_{-3}$}
\psfrag{h}{}
\psfrag{i}[c][c]{$-3$}
\psfrag{j}[c][c]{$-2$}
\psfrag{k}[c][c]{$-1$}
\psfrag{n}[c][c]{$n=0$}
\psfrag{l}[c][c]{$1$}
\psfrag{m}[c][c]{$2$}
\psfrag{p}[c][c]{$3$}
\psfrag{q}[l][l]{$\popfunc_2=\lfloor \oo+\nu\,\mu_{0..2}\rfloor-\lfloor \oo+\nu\, \mu_{0..1}\rfloor$}
\psfrag{r}[l][l]{$\popfunc_1=\lfloor \oo+\nu\, \mu_{0..1}\rfloor-\lfloor \oo+\nu\, \mu_{0}\rfloor$}
\psfrag{s}[l][l]{$\popfunc_0=\lfloor \oo+\nu\, \mu_{0}\rfloor-\lfloor \oo \rfloor$}
\psfrag{t}[l][l]{$\popfunc_{-1}=\lfloor \oo\rfloor-\lfloor \oo-\nu\, \mu_{-1}\rfloor$}
\psfrag{u}[l][l]{$\popfunc_{-2}=\lfloor \oo-\nu\, \mu_{-1}\rfloor-\lfloor \oo-\nu\, \mu_{-2..-1}\rfloor$}
\psfrag{v}[l][l]{$\popfunc_{-3}=\lfloor \oo-\nu\, \mu_{-2..-1}\rfloor-\lfloor \oo-\nu\, \mu_{-3..-1}\rfloor$}
\psfrag{w}{$p_\alpha(0)$}
\psfrag{P}{$\psi$}
   
 \includegraphics[width=4in]{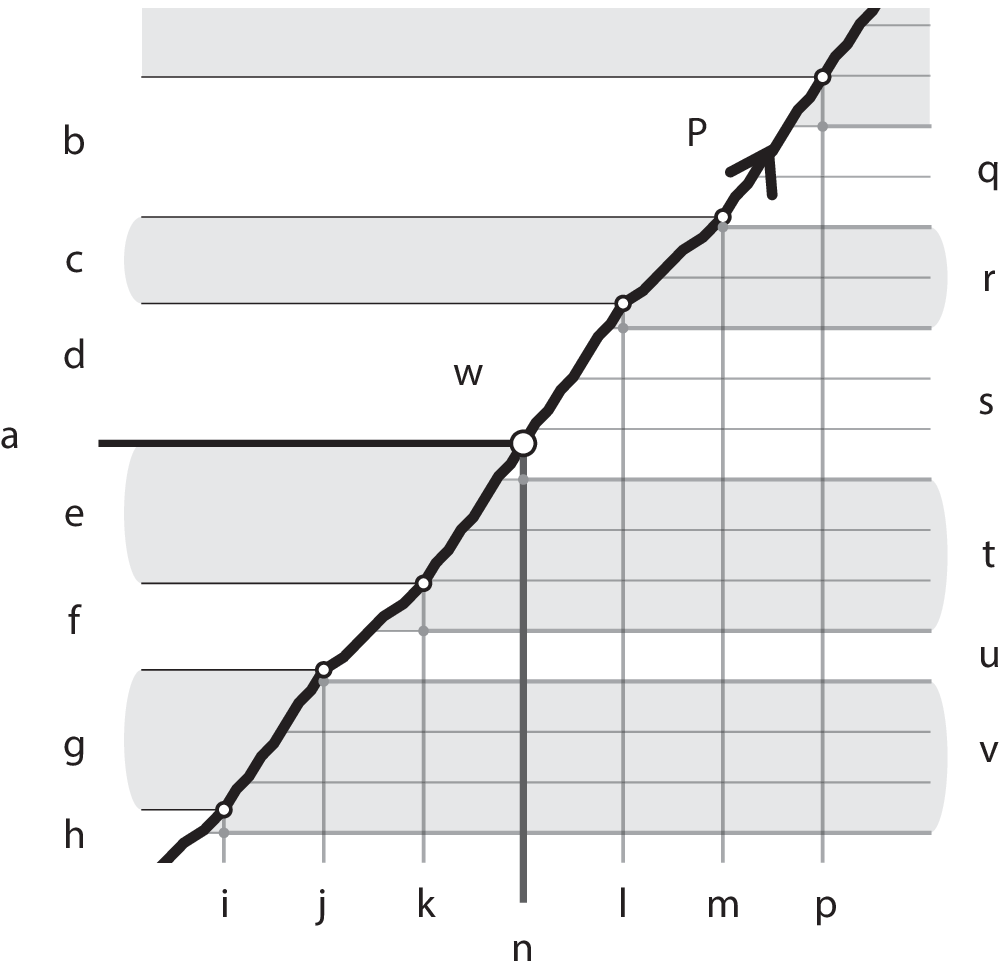}

$$\popfunc(v)=\popfunc(p_\alpha(n))=\popfunc_n := \left\{\begin{array}{ll}
\lfloor \oo+\nu\,\mu_{0..n}\rfloor-\lfloor \oo+\nu\,\mu_{0..(n-1)}\rfloor  & n>0\\ \\
\lfloor \oo+\nu\,\mu_{0}\rfloor-\lfloor \oo \rfloor  & n=0 
\\ \\ \lfloor \oo\rfloor-\lfloor \oo-\nu\,\mu_{-1}\rfloor & n = -1
\\ \\
 \displaystyle 
 \lfloor \oo-\nu\,\mu_{(n+1)..-1}\rfloor-\lfloor \oo-\nu\,\mu_{n..-1}\rfloor & n<-1 \end{array}\right.$$

 Note that for any $v\in G$, $\popfunc(v)$ has the form $$(\lfloor \nu\;\mu(v)+x\rfloor-\lfloor x\rfloor)\in\{0\}
\cup[\lfloor\nu\rfloor..\lceil\nu\max_{a\in\Aa}\mu(a)\rceil]$$ for some $x\in\Rr$.

By telescoping, along any finite interval of an orbit under $\psi$,

 $$ \left|  \popfunc_{m..n} - \nu\ \mu_{m..n}\right| < 2$$

\end{sumsubscripts}

We observe that $R\cap H^+$ is the disjoint union of maximal  sets of the form $p_\alpha(a..b)$.  
Because $\psi$ is $L$-Lipshitz,  between $p_\alpha(a)$ and $p_\alpha(a-1)$  the distance in the 
divergence graph is at most $L$ and so $p_\alpha(a-1)\in\partial R$.
 
 \vspace{\baselineskip}
 \centerline{\includegraphics[scale=.4]{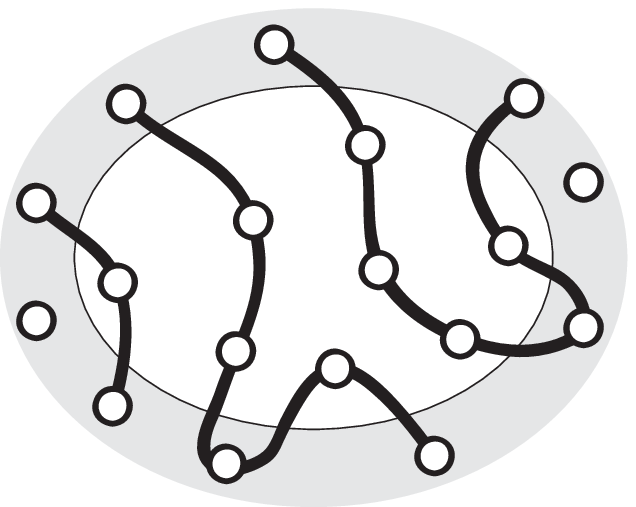}}
 
Consequently,  as indicated in the figure above,  there are at most $\#\partial R\cap H^+$ such maximal $p_\alpha(a..b)$ covering $R$, 
each contributing at most 2 to the error of $\popfunc$.  Since $\#\partial R\cap H^+\leq\mu(\partial R)$, we have that  \begin{sumsubscripts} $$ \left|\popfunc_{R} - \nu\ 
{
 \mu_{R}}\right|\leq 2  \mu_{\vspace{.1em}\partial R}$$ \end{sumsubscripts} and $\popfunc$ realizes $\nu$ up to error 2. 
\end{proof}

\begin{definition}
Given  a sequence $(\nu_i)\subset \Rr^{\Zz}$, a function $f:G\to\Nn$ {\em realizes $(\nu_i)$ up to error $C$} if for any $i\in\Zz$, the restriction of $\popfunc$ to $H_i$ realizes $\nu_i$ up to error $C$.
\end{definition}

We thus interpret Lemma ~\ref{densityOnHorofunctions} as:

\begin{cor}\label{corollary:ErrorRealizable} For any $A\geq 1$,  any sequence  
$(\nu_i)\in [A,qA]^\Zz$  is realized up to error 2 by some function $\displaystyle \popfunc:G\to \{0\}\cup[\lfloor A \rfloor.. \lceil qA\max_{a\in\Aa}\mu(a)\rceil]$.
\end{cor}

\begin{definition}
\label{defineSturmianesque}
For any fixed $A\geq 1$, we say that a  sequence   $(\nu_i, \density_i)_{i\in\Zz}$ is  {\em balanced} 
if it satisfies
$\displaystyle \nu_{i+1}=\frac{q^{\density_i}}{\lambda}\nu_i$ where 
$$\density_i=\left\{\begin{array}{ll}
 \lceil \log_q\lambda \rceil & \textrm{for $\nu_i\in [A,\frac{\lambda}{q^{\lfloor \log_q\lambda\rfloor} }A)$}\\
  \lfloor \log_q\lambda \rfloor & \textrm{for $\nu_i\in[\frac{\lambda}{q^{\lfloor \log_q\lambda\rfloor} }A,qA)$}  
 \end{array}\right.$$

\end{definition}

For any $\nu_0\in[A,qA)$, note there is  a unique balanced sequence $(\nu_i, \density_i)_{i\in\Zz}$.

In any balanced sequence, we drop the subscript if the context is clear.

The following Proposition ensures that local errors in the distribution of populations may be redistributed from horosphere to horosphere within  bounded domains.

\begin{prop}\label{prop:Balancing}Suppose that $\displaystyle \lfloor A\rfloor >(2q+2)\max_{a\in \Aa}\mu(a)$. For any balanced sequence $(\nu_i,\density_i)$, and $\displaystyle \popfunc:G\to \{0\}\cup[\lfloor A \rfloor.. \lceil qA\max_{a\in\Aa}\mu(a)\rceil]$ realizing $(\nu_i)\in [A,qA]^\Zz$   up to error 2, 
 there exists a bijection 
$$\Psi:\{(v,j,k):v\in H_i,\ j\in [1..\popfunc_v],k\in [1..q^{\density_i}]\}\leftrightarrow
 \{(u,l):u\in H_{i+1}, l\in [1..\popfunc_u]\}$$ 
 such that if $\Psi(v,j,k)=(u,l)$ then $P(u)\in \nbhd\{v\}$. \end{prop}

 \begin{proof}
We begin by describing a technique for producing bijections like the one we want. Let $\mathcal{G}$ be a locally finite bipartite graph with vertex partition $M\sqcup W$. A perfect matching for $\mathcal{G}$ is a collection $\mathcal{M}$ of edges of $\mathcal{G}$ such that every vertex of $\mathcal{G}$ belongs to exactly one edge from $\mathcal{M}$ \cite[\S H.2]{ccs}. We say that $\mathcal{G}$ satisfies the Hall conditions \cite[Definition H.3.1]{ccs} if for every finite $R$ which is a subset of $M$ or $W$, the set of vertices which are $\mathcal{G}$-neighbors of $R$ is at least as big as $R$. By \cite[Theorem H.3.6]{ccs}, $\mathcal{G}$ admits a perfect matching if and only if it satisfies the Hall conditions. 

In our case, we take $M$ to be $\{(v,j,k):v\in H,\ j\in [1..\popfunc_v],k\in [1..q^{\density}]\}$ and $W$ to be $\{(u,l):u\in H_{i+1}, l\in [1..\popfunc_u]\}$, with an edge of $\mathcal{G}$ connecting $(v,j,k)$ and $(u,l)$ whenever $v\in\nbhd P^{-1}\{u\}$. If $\mathcal{M}$ were a perfect matching for $\mathcal{G}$, then we could define the desired bijection $\Psi$ by taking $\Psi(v,j,k)$ to be the unique vertex $(u,l)$ of $W$ such that $(v,j,k)$ and $(u,l)$ span an edge of $\mathcal{M}$. It follows that we only need to verify that $\mathcal{G}$ satifies the Hall conditions.

For any $v\in H$, any  $(v,j,k),(v,j',k')\in M$ have the same $\mathcal{G}$-neighbors in $W$. Hence, if $R$ is a subset of $M$, then the number of $\mathcal{G}$-neighbors of $R$ (in $W$) is equal to the number of $\mathcal{G}$-neighbors of
$$\{(v,j,k)\in M:\exists (v,j',k')\in R\}\supset R.$$
Similar considerations apply when we wish to bound the number of $\mathcal{G}$-neighbors of a finite $R\subset W$. It follows that we only need to establish, for finite $R\subset H$, that 
\begin{sumsubscripts}$$q^{\density} \ \popfunc_R \leq  \popfunc_{P^{-1}\nbhd R}$$\end{sumsubscripts}
and for finite $R\subset H_{i+1}$ that
\begin{sumsubscripts}$$q^{-\density}\popfunc_{R}\leq \popfunc_{\nbhd PR}$$\end{sumsubscripts}
{In fact, for $R\subset H_{i+1}$ we have that $R\subseteq P^{-1}PR$ and $\nbhd PR=\nbhd PP^{-1}PR$, so we only need to check the latter inequality for sets of the form $P^{-1}PR$}.
 \begin{sumsubscripts}

Intuitively, $\popfunc_R$ is close to $\nu\,\mu_R$ and
 \begin{sumsubscripts}$\popfunc_{P^{-1}R}$\end{sumsubscripts} is close to $q^{\density}\nu\,\mu_R$, with the error controlled by $\mu_{\partial R}$, so we must show that \begin{sumsubscripts}$ \popfunc_{P^{-1}\partial R}$\end{sumsubscripts} is large enough to accomodate this error (because $A$ was chosen suitably large).

To show our desired inequalities,  we will need the following  identities on any finite $R,T\subset H$:

(1)  $\popfunc_R{\leq}2\mu_{\partial R}+\nu\,\mu_R$ and (1') $\nu\,\mu_R{\leq}2\mu_{\partial R}+\popfunc_R$,  because by Lemma~\ref{corollary:ErrorRealizable}, $\popfunc$ realizes $\nu\,\mu$ up to error 2. 
\\ \\

(2) $\mu_T=\frac{1}{\lambda}\,\mu_{P^{-1}T}$ and (2') $\lambda\,\mu_{PT}\geq \mu_T$: 
From the definition of $\mu$ and shortlex shelling, we have that
$$\sum_{P(b)=a} \mu(b) = \lambda \mu(a)$$ 
holds, giving~(2) directly.  For~(2') observe that, in addition, $\mu_{P^{-1}PT}\geq \mu_T$ holds.
\\ \\

(3) $\displaystyle q^\density/ \lambda < q$ and  (3') $\displaystyle   \lambda q^{-\density}< q$ by the definition of each $\density$ in a balanced sequence.
\\ \\
 If $v,w\in H_{i+1}$ are connected by a divergence edge, then $\dist(P^{-n}\{v\},P^{-n}\{w\})=O(1)$ and so  $\dist(P^{-n}\{P(v)\},P^{-n}\{P(w)\})=O(1)$. Therefore $P(v)$ and $P(w)$ either coincide or are connected by a divergence edge. It follows that   $\nbhd P^{-1} R\subset P^{-1} \nbhd  R$ and $\nbhd PR\supset P \nbhd  R$. 

\vspace{\baselineskip}
{
\psfrag{b}{$\in P^{-1}R$}
\psfrag{a}{$\in R$}
\psfrag{c}{$\in \nbhd P^{-1}R\subset P^{-1}\nbhd R$}
\psfrag{d}{$\in \nbhd R$} \includegraphics[width=1.5in]{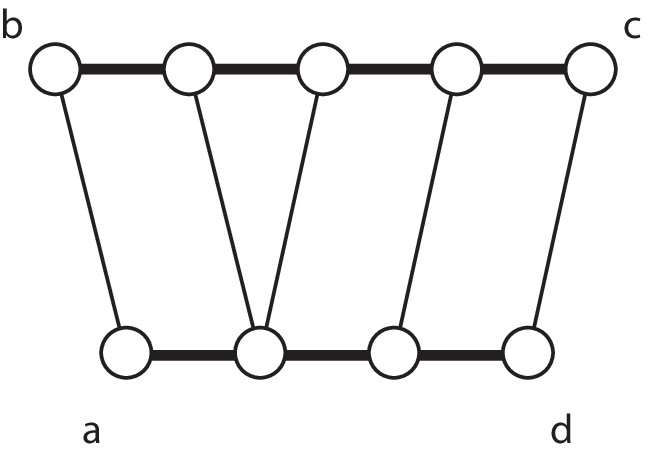}
\hspace{2in}
\psfrag{b}{$\in R$}
\psfrag{a}{$\in PR$}
\psfrag{c}{$\in \nbhd R$}
\psfrag{d}{$\in \nbhd P R\subset P\nbhd R$} \includegraphics[scale=.6]{HGSFT_diagram2}
}

As a consequence of the first inclusion (4) $\partial P^{-1}R\subset P^{-1}\partial R$. The second inclusion implies   $P\partial R\subset \nbhd P R$. If in addition if $R$ is ``sibling closed'', satisfying
$R=P^{-1}PR$, then we have (4') $\partial P R\supset P\partial R$. Sibling closed is necessary as indicated in the following diagram:

\vspace{\baselineskip}
\centerline{
\psfrag{a}[r][r]{$ \partial P R \not\ni $}
\psfrag{d}{$\in P\partial R\cap \nbhd P R$}
\psfrag{b}{$\in R$}
\psfrag{c}{$\in \partial R$}
 \includegraphics[scale=.6]{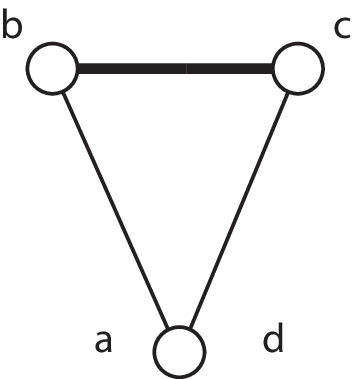}}

(5)  $(2q+2)\mu_T \leq \popfunc_T$ since for any  $v\in H^+$, $\popfunc(v)\geq \lfloor{\nu}\rfloor\geq A\geq(2q+2)\max_{\Aa}\mu(a)$ and $\popfunc(v)=0=\mu(v)$ otherwise.

We define $\nu':=\nu_{i+1}$, so that $\nu'={q^\density}/{\lambda}\nu$ and conversely $\nu={\lambda}{q^{-\density}}\nu'$.

For the $\rightarrow$ map, we need, for finite $R{\subset} H$, that
 $$q^{\density} \  \popfunc_{R} \leq  \popfunc_{P^{-1}\nbhd R} $$

$$
\begin{array}{rcll}
q^{\density} \popfunc_R &\leq & q^{\density}(2\mu_{\partial R}+\nu\,\mu_R) & \text{\hspace{1in} by (1).}
\\ \\
& = & \displaystyle \frac{q^{\density}}{\lambda}(2\mu_{P^{-1}\partial R}+\nu\,\mu_{P^{-1}R}) &\text{\hspace{1in} by (2).}
\\ \\
& \leq & 2q \,\mu_{P^{-1}\partial R}+\nu'\,\mu_{P^{-1}R} &\text {\hspace{1in} by (3) and definition of $\nu'$.}
\\ \\
& \leq & 2q\,\mu_{P^{-1}\partial R}+2\,\mu_{\partial P^{-1}R} +\popfunc_{P^{-1}R} &\text{\hspace{1in} by (1').}
\\ \\ 
& \leq & 2q\,\mu_{P^{-1}\partial R}+ 2\,\mu_{P^{-1}\partial R} +\popfunc_{P^{-1}R} & \text{\hspace{1in} by (4).}
\\ \\
& \leq & \popfunc_{P^{-1}\partial R}+\popfunc_{P^{-1}R} & \text{\hspace{1in} by (5).}
\\ \\ 
& = & \popfunc_{P^{-1}\nbhd R} & \text{\hspace{1in} as desired.}
\end{array}$$

To find an injection in the other direction, we need that  for finite $R\subset H'$

$$q^{-\density} \popfunc_R\leq \popfunc_{\nbhd P R}$$

We replace $R$ with its sibling closure $P^{-1}PR$;  the left hand side of the inequality cannot decrease and the right does not change, thus establishing the inequality for all $R$. 
We compute:

$$
\begin{array}{rcll}
q^{-\density} \popfunc_R&\leq &q^{-\density}(2\,\mu_{\partial R}+\nu'\,\mu_R)&\text{\hspace{1in} by (1).}
\\ \\
&\leq& q^{-\density}(2\lambda\, \mu_{P\partial R}+\nu'\lambda\,\mu_{PR})  &\text{\hspace{1in} by (2').}
\\ \\
&\leq& 2\lambda q^{-\density} \,\mu_{\partial PR}+\nu\,\mu_{PR}    &\text{\hspace{1in} by (4') and definition of $\nu$.}
\\ \\
&\leq&2q\mu_{\partial PR} + 2\,\mu_{\partial PR}+\popfunc_{PR}&\text{\hspace{1in} by (3') and (1').}\\ \\
&\leq& \popfunc_{\partial PR}+\popfunc_{PR}  &\text{\hspace{1in} by (5).}\\ \\
&=& \popfunc_{\nbhd P R}  &\text{\hspace{1in} as desired }
\end{array}$$

\end{sumsubscripts}

This completes the proof of Proposition~\ref{prop:Balancing}
\end{proof}

\begin{prop}\label{Cor:PopShellingsExist} For some $N$, there exists a populated shelling $X$ with population bound $N$ and growth by powers of $q$, and so the SFT $\Sigma$ is non-empty. 
\end{prop}

\begin{proof} Take  $\displaystyle A>(2q+2)\max_{a\in\Aa}\mu(a)$ and $\displaystyle N>\lceil qA\max_{a\in\Aa}\mu(a)\rceil$.
\end{proof}

\section{Aperiodicity}
\label{subsection:aperiodicity}

Any infinite hyperbolic group admits a  shortlex shelling $X$ such that $\local X$ is periodic  --- for example take a horofunction 
with axis defined by a cycle in a shortlex FSA.  By contrast, Proposition \ref{Prop:Aperiodic} shall show that for a populated shelling $X$ on a one-ended hyperbolic group, $\local X$ cannot have an infinite order period, completing the proof of our main theorem. The idea is to show that any period of $\local X$ would induce a period of the ``growth sequence'' $\density_i:=\density(h^{-1}(i))$ (this follows from Lemma \ref{lemma:shortlexPeriodic}), and then show that periods of the growth sequence cannot exist (Corollary \ref{Cor:DensityNonperiodic}).

We begin by showing that any infinite order period, say 
$\local{X} \cdot \pi = \local{X}$ for $\pi \in G$, translates horospheres to
horospheres and does not fix any horosphere:

\begin{lem}\label{lemma:shortlexPeriodic}
Given a shortlex shelling $X=(h,\state,P)$, if $\local{X}$ is periodic under some 
infinite order element $\pi\in G$, then $h(\pi g)=h(g)+C_\pi$ for some nonzero constant 
$C_\pi\in\Zz$.
\end{lem}

\begin{proof}
Write $h\cdot \pi$ for $g\mapsto h(\pi g)$, so that
$$\deriv (h\cdot \pi)=(\deriv h)\cdot \pi=\deriv h.$$ 
We see that $h$ and $h\cdot \pi$ differ by a constant, i.e., there is some $C_\pi\in\Zz$ such that $h(\pi g)=h(g)+C_\pi$.

We may see that $C_\pi\neq 0$ as follows. Without loss of generality, let $0=h(1_G)$. 
 If $C_\pi=0$, then $\ldots,\pi^{-1},1_G,\pi,\pi^2,\ldots$ is a quasi geodesic lying in a horosphere~\cite[Corollary III.$\Gamma$.3.10]{BridsonHaefliger99}. Hence, there must be some geodesic $\gamma:\Zz\To G$ such that $\gamma(\Zz)$ is at finite Hausdorff distance $N$ from this quasi geodesic, and in particular, $h\circ\gamma$ attains only values in $[-N..N]$. For any $R$ we may find a geodesic $\gamma$ of length $R$ in the $N$-neighborhood of $h^{-1}(0)$, connecting $1_G$ to some $g'$.   For $R>2N+2\delta$, by Lemma~\ref{lemma:DipSomeMore}, no such geodesic exists. 
 \end{proof}
 
\begin{lem} \label{LemmaDesc} Let $X=(h,\state,P,\density,m)$ be a populated shelling.  For  any horosphere $H$ and any $v\in H^+$, there is some finite  $S\subset H$ such that all the descendants of villagers in $v$ lie in $P^{-*}S$, the future of $S$. Furthermore there 
 is some finite  $S'\subset H$ so that every villager living in $P^{-*}S$ is descended from a villager living in $S'$.
\end{lem}

\begin{proof}
We write $\pi_G$ for the projection  from $G\times\Nn\to G$. Write $Q(u,l)=(v,j)$ where $(v,j,k)$ is the unique triple such that $m(v,j,k)=(u,l).$

We will show that there is a universal constant $R$ so that for any villager $(u,l)$ and $n\in\Nn$, we have $\dist(\pi_G(Q^{n}(u,l)), P^n(u))\leq 2R$. The proposition will follow: For any $v$ in any $H^+$, take $S$ to be the $2R$-neighborhood of $v$ and let $S'$ be the $2R$-neighborhood of $S$.

Suppose that $(u,l)$ is a descendant of a villager at $v$, i.e., that $\pi_G(Q^n(u,l))=v$ for some $n>0$ and $l\in [1..\popfunc(u)]$. Let $v'=P^n(u)$ and take $B$ be a ball containing $\{P^{k}u\}_{k=0}^n \cup \{\pi_G Q^k(u,l)\}_{k=0}^n$.

By the definition of  a shortlex shelling, 
$B$ is modeled in $X_0$; that is there exists $g\in G$ such that 
$$\local{X_0}\cdot g|_{v^{-1}B}
=\local(h,\state,P)\cdot v|_{v^{-1}B}$$

\centerline{
\psfrag{g}{$\gamma'$} 
\psfrag{h}{$\gamma$} 
\psfrag{w}{$v'$}
\psfrag{v}{$v$}
\psfrag{r}{$R$}
\psfrag{u}{$u$}
\psfrag{y}{$v''$}
\psfrag{z}{$vg^{-1}$}
\includegraphics[scale=.6]{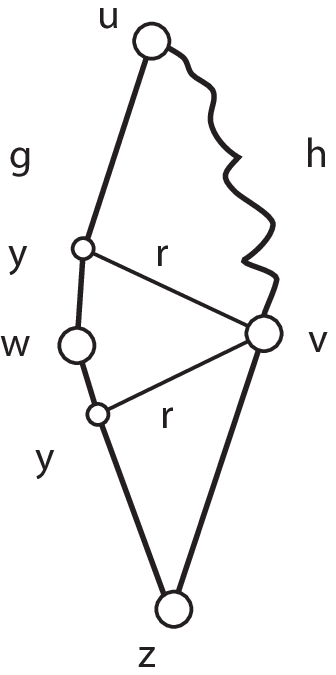}}
Let $\gamma'$ be the geodesic given by $\gamma'(k)=vg^{-1}P_0^k(gv^{-1}u)$, $k\in [0..(n+|g|)]$.

 Let $\gamma$ be the path defined by:
 For $k\in [0..n]$, take $\gamma(k)=\pi_G Q^k(u,l)$.  For $k\in [n..(n+|g|)]$, $\gamma(k)=vg^{-1}P_0^k(g)$.
 
 Observe that
$h_0(gv^{-1}\gamma(k))=h_0(gv^{-1}u)-k$ and therefore $\dist(\gamma(i),\gamma(j))\geq |i-j|$.

We claim that 
$\dist(\gamma(k),\gamma(k+1))\leq 2\delta L+1$. 
For $k< n$, edges of the divergence graph have length at most $2\delta$ (Lemma~\ref{LemmaEdgesOfLength2Delta}) 
and (by Definition~\ref{definition:populatedshelling} of a populated shelling) $m_G(v,j,k)$ lies in $P^{-1}$ of the $L$-neighborhood  of $v$
 in the divergence graph in the horosphere $h^{-1}(h(v))$. For  $k\geq n$, $\dist(\gamma(k)$,$\gamma(k+1))=1$.

It follows that 
$$|i-j|=|h(\gamma(i))-h(\gamma(j))|
\leq \dist(\gamma(i),\gamma(j))\leq (2\delta L +1)|i-j|$$
for $i,j\in [0..(n+|g|)]$.
Then  $\gamma$ is a {$(2\delta L+1,0)$}-quasi geodesic, as defined in~\cite[Definition I.8.22]{BridsonHaefliger99}.

We have that  $\gamma(0)=\gamma'(0)=u$ and $\gamma(n+|g|)=\gamma'(n+|g|)=vg^{-1}$. 
By Theorem III.1.7 of~\cite{BridsonHaefliger99} every point of $\gamma$ is within a universal bound, which is 
denoted there by $R=R(\delta,2\delta L+1,0)$, of
some point of $\gamma$.  In particular, $v \in \gamma$ is within $R$ of some point
$v'' \in \gamma'$.  By the triangle inequality,
$\dist(v',v'') \leq \dist(v'',v)$.  Combining these facts we conclude that  
$\dist(v,v') \leq 2R$.
\end{proof}

{Recall that, if $X=(h,\state,P,\popfunc,\density,\match)$ is a populated shelling, then $\density$ is 
constant on horospheres. Write $\Delta_i$ for the value achieved by $\density$ on $h^{-1}\{i\}$.  
We will refer to $(\Delta_i)_{i\in\Zz}$ as the growth sequence of $X$.}

\begin{cor}
In a populated shelling, there exists a non-empty finite set $S\subset H^+$ such that we have the following.
\begin{enumerate}
\begin{sumsubscripts}
\item $\log( \popfunc_{P^{-n}S})=n\log(\lambda)+O(1)$
\item $\log( \popfunc_{P^{-n}S})=\sum_{i=1}^n \density_i\log(q)+O(1)$.
\end{sumsubscripts}
\end{enumerate}
\end{cor}

\begin{proof}
\begin{sumsubscripts}
Note $\mu_{P^{-n}S}  = \lambda^n\mu_S$.  The functions $\mu$ and $\popfunc$ have finite non-negative ranges, 
and have identical zero-sets. Therefore, there are constants $c_1,c_2>0$ such that for any $v\in G$, 
$c_1 {\popfunc(v)} \leq{\mu(v)}\leq c_2 {\popfunc(v)}$. Consequently 
$ c_1 {\popfunc_{P^{-n}S}} \leq{\mu_{P^{-n}S}}\leq c_2 {\popfunc_{P^{-n}S}}$ and the first equality (1) follows. 

By Lemma~\ref{LemmaDesc} there exists some $v$ such that all the descendants of villagers in $v$ lie in $P^{-*}S$, 
the future of $S$ and so $\popfunc_{P^{-n}S}\geq \left(q^{\sum_1^n{\density_i}}\right)\popfunc_v$, the number of 
such descendants. 

Lemma~\ref{LemmaDesc}  further shows there is some finite $S'\in H$ so that every villager living in $P^{-*}S$ is descended 
from a villager living in $S'$, and so $\popfunc_{P^{-n}S}\leq q^{\sum_1^n{\density_i}}\popfunc_{P^{-n}S'}$.  
Together these inequalities give (2).
\end{sumsubscripts}
\end{proof}

\begin{cor} \label{Cor:DensityNonperiodic}
The growth sequence in a populated shelling is not periodic. 
\end{cor}

\begin{proof} Suppose the growth sequence sequence $(\density_i)$ is periodic, with period $p\in\Nn$.  
Let $\hat\density=\sum_{i=1}^{p} \density_i$. For any $k\in\Nn$, taking $n=pk$, we have 
$pk\log(\lambda)+O(1)=\hat\density k\log(q)$ and thus
$p\log(\lambda)+O(1/k)=\hat\density\log(q)$.
As $k\to\infty$, $$\log(q)/\log(\lambda)=p/\hat\density\in\Qq$$ a contradiction to our choice of $q$ with respect to $\lambda$.

\end{proof}

\begin{prop}\label{Prop:Aperiodic}{Let $\local X$ be the local data for  a populated shelling 
$X=(h,\state,P,\popfunc,\density,\match)$. Then the stabilizer of  $\local X$ in $G$ contains no infinite order element.}
\end{prop}

\begin{proof}
{Suppose $\pi$ is in the stabilizer of $\local X$, so that $\local X\cdot \pi=\local X$, and  $\pi$ 
has infinite order. By Lemma~\ref{lemma:shortlexPeriodic}, there is a nonzero $C_\pi\in\Zz$ 
such that $h(\pi g)=h(g)+C_\pi$. Writing $(\density_i)$ for the growth sequence of $X$, it follows 
that $\density_{h(g)}=\density_{h(g)+C_\pi}$ for every $g\in G$, and hence $(\density_i)$  is 
$C_\pi$-periodic, in contradiction to Lemma~\ref{Cor:DensityNonperiodic}.}\end{proof}

%
%

\newcommand{\etalchar}[1]{$^{#1}$}

\end{document}